\documentclass{my-aims}

\usepackage{amsmath}
\usepackage{paralist}
\usepackage{graphics} 
\usepackage{epsfig} 
\usepackage{graphicx}  
\usepackage{epstopdf}
\usepackage{subfig}
\usepackage[colorlinks=true]{hyperref}
\hypersetup{urlcolor=blue, citecolor=red}

\usepackage{url}
\usepackage{breakurl}

\def\ppages{X--XX}

\newtheorem{theorem}{Theorem}[section]

\newtheorem{lemma}[theorem]{Lemma}

\theoremstyle{definition}

\newtheorem{remark}{Remark}

\def\R{\mathbb{R}}

\title[Modeling and optimal control of HIV/AIDS]{Modeling 
and optimal control of HIV/AIDS prevention through PrEP}

\author[C. J. Silva and D. F. M. Torres]{}



\email{cjoaosilva@ua.pt}
\email{delfim@ua.pt}

\thanks{$^*$Corresponding author: Cristiana J. Silva (cjoaosilva@ua.pt)}

\begin{document}

\maketitle

\centerline{\scshape Cristiana J. Silva$^*$ and Delfim F. M. Torres}
\medskip
{\footnotesize
\centerline{Center for Research \& Development in Mathematics and Applications (CIDMA)}
\centerline{Department of Mathematics, University of Aveiro}
\centerline{3810--193 Aveiro, Portugal}}

\bigskip


\begin{abstract}
Pre-exposure prophylaxis (PrEP) consists in the use of an antiretroviral 
medication to prevent the acquisition of HIV infection by uninfected 
individuals and has recently demonstrated to be highly efficacious 
for HIV prevention.	We propose a new epidemiological 
model for HIV/AIDS transmission including PrEP.  
Existence, uniqueness and global stability of the 
disease free and endemic equilibriums are proved. 
The model with no PrEP is calibrated with the cumulative 
cases of infection by HIV and AIDS reported in Cape Verde 
from 1987 to 2014, showing that it predicts well such reality. 
An optimal control problem with a mixed state control 
constraint is then proposed and analyzed,
where the control function represents the PrEP strategy 
and the mixed constraint models the fact that, due to PrEP costs, 
epidemic context and program coverage, the number of individuals 
under PrEP is limited at each instant of time. The objective 
is to determine the PrEP strategy that satisfies the mixed state 
control constraint and minimizes the number of individuals 
with pre-AIDS HIV-infection as well as the costs associated 
with PrEP. The optimal control problem is studied analytically. 
Through numerical simulations, we demonstrate that PrEP 
reduces HIV transmission significantly.
\end{abstract}
	
	
\section{Introduction}

The human immunodeficiency virus (HIV) is a retrovirus that causes 
HIV infection and, over time, acquired immunodeficiency syndrome (AIDS) 
\cite{Weiss:1993:Science}. The most significant advance in medical 
management of HIV infection has been the treatment of patients 
with antiviral drugs, which can suppress HIV replication to undetectable levels.
Before 1996, few antiretroviral (ART) treatment options for HIV infection existed. 
The treatment of HIV infection was revolutionized in the mid 1990s 
with the introduction of drug regimens that combine inhibitors 
of the reverse transcriptase and protease and two of three essential enzymes of HIV. 
Combination of antiretroviral drugs, dramatically suppresses viral replication 
and reduces the HIV viral load to levels below the limits of detection 
of the most sensitive clinical assays, resulting in a significant 
reconstitution of the immune system \cite{Eric:HIV:2012}.
 
The \emph{Global AIDS Update 2016} of the Joint United Nations Programme 
on HIV/AIDS, reports that the global coverage of ART therapy reached 
approximately 46\% at the end of 2015. The gains in treatment 
are largely responsible for a 26\% decline in AIDS-related deaths 
globally since 2010, from an estimated 1.5 million in 2010, 
to 1.1 million in 2015. Despite this significant achievement, 
there has not been a decrease in new infections since 2010, 
with more than 2 million new infections reported in 2015 
\cite{UNAIDS:Global:AIDS:2016}. The World Health Organization's (WHO) 
Global Health Sector Strategy on HIV embraces innovation in the HIV response, 
recommending, for example, that people at substantial risk of HIV infection 
should be offered pre-exposure prophylaxis (PrEP) as an additional 
prevention choice, as part of comprehensive prevention. PrEP is the use 
of an antiretroviral medication to prevent the acquisition of HIV infection 
by uninfected individuals. Several trials among men who have sex with men, 
people who inject drugs, transgender people, women and serodiscordant 
couples (one partner is HIV-positive and the other is HIV-negative)
have shown that when PrEP is taken, it is an effective and safe 
mechanism for preventing HIV-infection \cite{WHO:PrEP}.

In \cite{Abbas:PrEP:2007}, it is concluded that PrEP could prevent 
2.7 to 3.2 million new cases of HIV in sub-Saharan Africa over 10 years, 
if it is targeted to the highest risk groups, and disinhibition 
could be prevented. In 2008, the authors of \cite{Gay:PrEP:2008} 
claimed that PrEP represents the most powerful available biologic intervention 
for HIV prevention. In 2016, WHO has welcomed a plan by the South African Ministry 
of Health to provide immediate antiretroviral treatment to all sex workers 
with HIV, and to offer daily oral PrEP to HIV-negative sex workers 
to prevent them from acquiring the infection \cite{url:WHO:SouthAfrica:PrEP}. 
Following WHO, making PrEP drugs available for safe, effective prevention 
outside the clinical trial setting is the current challenge. 

Substantial gaps remain in understanding the trade-offs between the costs 
and benefits of choosing alternative HIV prevention strategies, 
such as the initiation of PrEP by high risk uninfected individuals \cite{Drabo:PreP:CID:2016}. 
Mathematical models of HIV that include PrEP are scarce \cite{Celum:JInfecDis:HIV:Prep}.
In \cite{Kim:ModelPrEP:PlosOne:2014}, a mathematical model is used to estimate 
the effects of early diagnosis, early treatment and PrEP, on the HIV epidemic 
in South Korea over the next 40 years, as compared with the current situation.
In \cite{Drabo:PreP:CID:2016}, the authors develop a mathematical model 
to simulate HIV incidence among men residing in Los Angeles County, CA, 
aged between 15 to 65 years old, who have sex with men. They claim that PrEP 
and Test-and-Treat yield the largest reductions in HIV incidence, and are highly 
cost-effective. Another cost-effectiveness study was done in \cite{Alistar:PrEP:2014} 
and the results showed that PrEP can be cost-saving if delivered 
to individuals at increased risk of infection. In this paper, 
we propose a new mathematical epidemiological model for HIV/AIDS transmission 
including PrEP, which generalizes the HIV/AIDS sub-model recently proposed 
in \cite{SilvaTorres:TBHIV:2015}. 

First, we consider the mathematical model with no PrEP, calibrate 
the model to the cumulative cases of infection by HIV and AIDS 
from 1987 to 2014 reported in \cite{report:HIV:AIDS:capevert2015}, 
and we show that the model predicts well the reality given in \cite{report:HIV:AIDS:capevert2015}. 
In this model, the effective contact with people infected with HIV includes 
two modification parameters that account for the relative infectiousness 
of individuals with AIDS symptoms, in comparison to those
infected with HIV with no AIDS symptoms \cite{art:viral:load}, 
and for partial restoration of the immune function of individuals 
with HIV infection that use ART correctly \cite{AIDS:chronic:Lancet:2013}. 
It should be noted that the infectiousness 
of the HIV-infected individuals under ART treatment 
is a controversial subject. HIV treatment reduces the viral load in the blood, 
and also in other body fluids, such as semen and vaginal fluid. However, 
not all people living with HIV who take HIV treatment and have an undetectable 
viral load in the blood also have an undetectable viral load in their other bodily fluids: 
see, e.g., \cite{Loutfy:PlosONE:2013} and references therein. In our paper, 
the values considered for the modification parameters are based on two different research studies: 
the first one is known as \emph{HPTN 052}, where it was found that the risk 
of HIV transmission among heterosexual serodiscordant couples is 96\% lower 
when the HIV-positive partner is on treatment \cite{Cohen:NEJM:2011}; 
and the other one, where it was proved that HIV-infected individuals 
under ART treatment have a very low probability (assumed inferior to $0.04$)
of transmitting HIV \cite{DelRomero:2016}. We prove the global stability 
of the endemic equilibrium and we provide numerical simulations, 
illustrating the calibration of the model to the HIV/AIDS situation in Cape Verde.

In \cite{Monteiro:PrEP:CV:2015}, the authors evaluate the effect of early HIV 
treatment and optimization of care, HIV testing, condom distribution, 
and substance abuse treatment on HIV incidence from 2011 to 2021, 
using Cape Verde as an example. However, they did not include PrEP 
for HIV negative groups at risk as a possible prevention measure. 
Here, we show that inclusion of PrEP can reduce significantly the HIV incidence. 
We start by proving existence and global stability of the disease free equilibrium 
of the HIV/AIDS-PrEP model. Moreover, we also prove existence of an unique endemic 
equilibrium and the global stability for some specific relevant cases. 
It is important, however, to highlight that PrEP is not for everyone \cite{url:aids:PrEP}. 
Only people who are HIV-negative and at very high risk for HIV infection should take PrEP. 
Moreover, PrEP is highly expensive and it is still not approved in many countries, e.g., 
by the European Medicines Agency (EMEA) \cite{Review:PrEP:Infection:2016}. Therefore, 
the number of individuals that should take PrEP is limited at each instant of time. 
In order to study this health public problem, we formulate an optimal control problem 
with a mixed state control constraint. 

Optimal control theory has been successfully applied to several
epidemiological models, e.g., dengue \cite{MyID:283,MyID:306},
tuberculosis \cite{MyID:314,MyID:353}, Ebola \cite{MyID:335,MyID:340},
cholera \cite{MyID:356}, and HIV/AIDS \cite{MyID:355,MyID:359}.
However, we claim our work to be the first to apply optimal control 
to an HIV/AIDS model with PrEP. More precisely, we consider the HIV/AIDS-PrEP model 
and formulate an optimal control problem with the aim to determine the PrEP strategy 
that satisfies the mixed state control constraint and minimizes the number of individuals 
with pre-AIDS HIV-infection as well as the costs associated with PrEP. We solve the optimal 
control problem and provide numerical simulations, which show that it is possible 
to reduce the HIV incidence through PrEP and having into consideration the limitations 
related to the implementation of PrEP (cost, epidemic context, program coverage 
and individual-level adherence).

The paper is organized as follows. In Section~\ref{sec:SICA:model}, 
we consider the SICA model for HIV/AIDS transmission proposed 
in \cite{SilvaTorres:TBHIV:2015} and prove the global stability 
of the unique endemic equilibrium for the case when 
the associated AIDS-induced mortality is negligible. 
We calibrate the SICA model to the cumulative cases of infection 
by HIV and AIDS from 1987 to 2014 in Cape Verde and show 
that it predicts well this reality. In Section~\ref{sec:SICAE:model}, 
we generalize the SICA model by including the possibility of providing PrEP 
to susceptible individuals. We prove the existence and global stability 
of the disease-free equilibrium for $R_0 < 1$, where $R_0$ denotes 
the basic reproduction number, which is computed following \cite{van:den:Driessche:2002}. 
We prove existence of an unique endemic equilibrium point for $R_0 > 1$ 
and its global stability for a negligible AIDS-induced death rate 
and strict adherence to PrEP. Through numerical simulations, 
we investigate the impact of PrEP in the reduction of HIV transmission. 
In Section~\ref{sec:optimal:control}, we propose and analyze an optimal control 
problem with a mixed state control constraint. Numerical simulations show that 
the extremal solutions combine a reduction of HIV transmission with limited 
number of individuals under PrEP at each instant of time. 
We end with Section~\ref{sec:conclusion} of conclusions and future work.  


\section{The SICA model for HIV/AIDS transmission}
\label{sec:SICA:model}

In this section, we analyze a mathematical model 
for HIV/AIDS transmission with varying population size 
in a homogeneously mixing population, first proposed 
in \cite{SilvaTorres:TBHIV:2015}, and prove the global 
stability of the unique endemic equilibrium 
for the case when the associated AIDS-induced mortality is negligible. 
The model subdivides human population into four mutually-exclusive 
compartments: susceptible individuals ($S$); 
HIV-infected individuals with no clinical symptoms of AIDS 
(the virus is living or developing in the individuals 
but without producing symptoms or only mild ones) 
but able to transmit HIV to other individuals ($I$); 
HIV-infected individuals under ART treatment (the so called 
chronic stage) with a viral load remaining low ($C$); 
and HIV-infected individuals with AIDS clinical symptoms ($A$).
The total population at time $t$, denoted by $N(t)$, is given by
$N(t) = S(t) + I(t) + C(t) + A(t)$.
Effective contact with people infected with HIV is at a rate $\lambda$, given by
\begin{equation*}
\lambda = \frac{\beta}{N} \left( I + \eta_C \, C  + \eta_A  A \right),
\end{equation*}
where $\beta$ is the effective contact rate for HIV transmission.
The modification parameter $\eta_A \geq 1$ accounts for the relative
infectiousness of individuals with AIDS symptoms, in comparison to those
infected with HIV with no AIDS symptoms. Individuals with AIDS symptoms
are more infectious than HIV-infected individuals (pre-AIDS) because
they have a higher viral load and there is a positive correlation
between viral load and infectiousness \cite{art:viral:load}. 
On the other hand, $\eta_C \leq 1$ translates the partial restoration 
of immune function of individuals with HIV infection 
that use ART correctly \cite{AIDS:chronic:Lancet:2013}.
All individuals suffer from natural death, at a constant rate $\mu$. 
We assume that HIV-infected individuals 
with and without AIDS symptoms have access to ART treatment. 
HIV-infected individuals with no AIDS symptoms $I$ progress to the class 
of individuals with HIV infection under ART treatment $C$ at a rate $\phi$, 
and HIV-infected individuals with AIDS symptoms are treated for HIV at rate $\alpha$.
We also assume that an HIV-infected individual with AIDS symptoms $A$ 
that starts treatment moves to the class of HIV-infected individuals $I$, 
moving to the chronic class $C$ only if the treatment is maintained. 
HIV-infected individuals with no AIDS symptoms $I$ that do not take 
ART treatment progress to the AIDS class $A$ at rate $\rho$. Note that 
only HIV-infected individuals with AIDS symptoms $A$ 
suffer from an AIDS induced death, at a rate $d$. These assumptions 
are translated in the following mathematical model:
\begin{equation}
\label{eq:model:1}
\begin{cases}
\dot{S}(t) = \Lambda - \frac{\beta \left( I(t) + \eta_C \, C(t)  
+ \eta_A  A(t) \right)}{N(t)} S(t) - \mu S(t),\\[0.2 cm]
\dot{I}(t) =  \frac{\beta \left( I(t) + \eta_C \, C(t)  
+ \eta_A  A(t) \right)}{N(t)} S(t) - (\rho + \phi + \mu)I(t) 
+ \alpha A(t)  + \omega C(t), \\[0.2 cm]
\dot{C}(t) = \phi I(t) - (\omega + \mu)C(t),\\[0.2 cm]
\dot{A}(t) =  \rho \, I(t) - (\alpha + \mu + d) A(t).
\end{cases}
\end{equation}
We consider the biologically feasible region
\begin{equation}
\label{Omega:inv:region:HIV}
\Omega = \{ \left( S, I, C, A \right) \in \R_+^{4} \, : \, N \leq \Lambda/\mu \}.
\end{equation}
Using a standard comparison theorem (see \cite{Lakshmikantham:1989}), 
one can easily show that $N(t) \leq \frac{\Lambda}{\mu}$ 
if $N(0) \leq \frac{\Lambda}{\mu}$. Thus, the region 
$\Omega$ defined by \eqref{Omega:inv:region:HIV} is positively invariant. 
Hence, it is sufficient to consider the dynamics 
of the flow generated by \eqref{eq:model:1} in $\Omega$.
In this region, the model is epidemiologically and mathematically 
well posed in the sense of \cite{Hethcote:2000}. In other words,
every solution of model \eqref{eq:model:1} with initial conditions
in $\Omega$ remains in $\Omega$ for all $t > 0$. For this reason,
the dynamics of the model is considered in $\Omega$.

\begin{theorem}[See \cite{SilvaTorres:TBHIV:2015}]
\label{theo:persistence:model:1}
The population $N(t)$ is uniformly persistent, that is,
$$
\liminf_{t \to \infty} N(t) \geq \varepsilon 
$$ 
with $\varepsilon > 0$ not depending on the initial data.
\end{theorem}

Model \eqref{eq:model:1} has a disease-free equilibrium, given by
\begin{equation*}
\Sigma_0 = \left( S^0, I^0, C^0, A^0  \right) 
= \left(\frac{\Lambda}{\mu},0, 0,0  \right).
\end{equation*}
Following \cite{van:den:Driessche:2002}, the basic reproduction 
number $R_{0}$ for model \eqref{eq:model:1}, which represents 
the expected average number of new HIV infections produced 
by a single HIV-infected individual when in contact with 
a completely susceptible population, is given by
\begin{equation}
\label{eq:R0:model:1}
R_0 = \frac{ \beta\, \left(  \xi_2  \left( \xi_1 +\rho\, \eta_A \right) 
+ \eta_C \,\phi \, \xi_1 \right) }{\mu\, \left(  \xi_2  \left( \rho + \xi_1\right) 
+\phi\, \xi_1 +\rho\,d \right) +\rho\,\omega\,d} 
= \frac{\mathcal{N}}{\mathcal{D}},
\end{equation}
where $\xi_1 = \alpha + \mu + d$, 
$\xi_2 = \omega + \mu$ and $\xi_3 = \rho + \phi + \mu$.

\begin{lemma}[See \cite{SilvaTorres:TBHIV:2015}]
The disease free equilibrium $\Sigma_0$ is locally asymptotically stable
if $R_0 < 1$, and unstable if $R_0 > 1$.
\end{lemma}

\begin{theorem}[See \cite{SilvaTorres:TBHIV:2015}]
The disease free equilibrium $\Sigma_0$ is globally 
asymptotically stable for $R_0 < 1$. 
\end{theorem}

To find conditions for the existence of an equilibrium for which HIV 
is endemic in the population (i.e., at least one of $I^*$, $C^*$ or 
$A^*$ is nonzero), denoted by $\Sigma_+ = \left(S^*, I^*, C^*, A^* \right)$, 
the equations in \eqref{eq:model:1} are solved in terms of the force 
of infection at the steady-state $\lambda^*$, given by
\begin{equation}
\label{eq:lambdaH:ast}
\lambda^* = \frac{\beta \left( I^* + \eta_C \, C^* + \eta_A A^*  \right)}{N^*}.
\end{equation}
Setting the right hand side of the equations of the model to zero, 
and noting that $\lambda = \lambda^*$ at equilibrium gives
\begin{equation}
\label{eq:end:equil:model:1}
S^* = \frac{\Lambda}{\lambda^* + \mu}\, , \quad 
I^* =-\frac{\lambda^* \Lambda \xi_1 \xi_2}{D} \, , \quad 
C^* =- \frac{\phi \lambda^* \Lambda \xi_1}{D} \, , \quad 
A^* = -\frac{\rho_1 \lambda^* \Lambda \xi_2}{D}
\end{equation}
with 
$D = -(\lambda^* + \mu)(\mu \left(\xi_2 (\rho + \xi_1) 
+ \xi_1 \phi + \rho d \right) + \rho \omega d)$, we use 
\eqref{eq:end:equil:model:1} in the expression for $\lambda^*$ 
in \eqref{eq:lambdaH:ast} to show that the nonzero (endemic) 
equilibrium of the model satisfies
\begin{equation*}
\lambda^* = -\mu (1-R_0).
\end{equation*}
The force of infection at the steady-state $\lambda^*$ is positive only if $R_0 > 1$. 
Thus, the existence and uniqueness of the endemic equilibrium follows. 

\begin{lemma}[See \cite{SilvaTorres:TBHIV:2015}]
The model \eqref{eq:model:1} has a unique endemic equilibrium whenever $R_0 > 1$.
\end{lemma}

\begin{remark}
The expression of the endemic equilibrium of model \eqref{eq:model:1} is given by
\begin{equation*}
\begin{split}
S^* &= \frac{ \Lambda (\rho d \xi_2 
- \mathcal{D})}{\mu (\rho d \xi_2 - \mathcal{N})}\, , \quad \quad 
I^* = \frac{\Lambda \xi_1 \xi_2 (\mathcal{D} 
-\mathcal{N})}{\mathcal{D} (\rho d \xi_2 -\mathcal{N} )} \, , \\
C^* &= \frac{\Lambda \phi \xi_1 (\mathcal{D} -\mathcal{N})}{\mathcal{D} 
(\rho d \xi_2 -\mathcal{N} )} \, , \quad \quad
A^* = \frac{\Lambda \rho \xi_2 (\mathcal{D} 
-\mathcal{N})}{\mathcal{D} ( \rho d \xi_2 -\mathcal{N} )}.
\end{split}
\end{equation*}
\end{remark}

\begin{theorem}[See \cite{SilvaTorres:TBHIV:2015}]
The endemic equilibrium $\Sigma_+$ is locally asymptotically 
stable for $R_0$ near 1.
\end{theorem}


\subsection{Global stability of the endemic equilibrium for negligible AIDS-induced death rate ($d=0$)}
\label{sec:glob:stab:SICA:d=0}

In this section, we investigate the global stability of the endemic equilibrium of model \eqref{eq:model:1} 
for the case when the associated AIDS-induced mortality is negligible ($d=0$). Adding the equations 
of the model \eqref{eq:model:1}, with $d=0$, gives $\dot{N} = \Lambda - \mu N$, 
so that $N \to \frac{\Lambda}{\mu}$ as $t \to \infty$. Thus, $\frac{\Lambda}{\mu}$ 
is an upper bound of $N(t)$ provided that $N(0) \leq \frac{\Lambda}{\mu}$. Further, 
if $N(0) > \frac{\Lambda}{\mu}$, then $N(t)$ decreases to this level. Using 
$N = \frac{\Lambda}{\mu}$ in the force of infection 
$\lambda = \frac{\beta}{N} \left( I + \eta_C \, C  + \eta_A  A \right)$ 
gives a limiting (mass action) system (see, e.g., \cite{FAgusto:bovineTB:2011}). 
Then, the force of infection becomes
$$
\lambda = \beta_1 \left( I + \eta_C \, C  + \eta_A  A \right) \, , 
\quad \text{where} \quad \beta_1 = \frac{\beta \mu}{\Lambda}.
$$
Therefore, we consider the model
\begin{equation}
\label{eq:model:2}
\begin{cases}
\dot{S}(t) = \Lambda - \beta_1 \left( I(t) + \eta_C \, C(t)  
+ \eta_A  A(t) \right) S(t) - \mu S(t),\\[0.2 cm]
\dot{I}(t) =  \beta_1 \left( I(t) + \eta_C \, C(t)  
+ \eta_A  A(t) \right) S(t) - \xi_3 I(t) + \alpha A(t) + \omega C(t), \\[0.2 cm]
\dot{C}(t) = \phi I(t) - \xi_2 C(t),\\[0.2 cm]
\dot{A}(t) =  \rho \, I(t) - \xi_1 A(t),
\end{cases}
\end{equation}
where $\xi_1 = \alpha + \mu$. The model \eqref{eq:model:2} has a unique endemic equilibrium 
given by $\tilde{\Sigma}_+ = \Sigma_+|_{d=0}$, whenever $\tilde{R}_0 = R_0|_{d=0} > 1$. 
Let us define
\begin{equation*}
\Omega_0 = \{ \left( S, I, C, A \right) \in \Omega \, : \, I=C=A=0 \}.
\end{equation*}

\begin{theorem}
The endemic equilibrium $\tilde{\Sigma}_+$ of model \eqref{eq:model:2} 
is globally asymptotically stable in $\Omega \backslash \Omega_0$ 
whenever $\tilde{R}_0 > 1$.  
\end{theorem}

\begin{proof}
Consider the following Lyapunov function:
\begin{equation}
\label{eq:Lyapunov:function}
V = \left(S - S^* \ln(S)\right) + \left(I - I^* \ln(I)\right) 
+ \frac{\omega}{\xi_2} \left(C - C^* \ln(C)\right) 
+ \frac{\alpha}{\xi_1} \left(A - A^* \ln(A)\right).
\end{equation}
Differentiating $V$ with respect to time gives
\begin{equation*}
\dot{V} = \left(1-\frac{S^*}{S}\right)\dot{S} + \left(1-\frac{I^*}{I}\right)\dot{I} 
+ \frac{\omega}{\xi_2} \left(1-\frac{C^*}{C}\right)\dot{C} + \frac{\alpha}{\xi_1} 
\left(1-\frac{A^*}{A}\right)\dot{A}.
\end{equation*}
Substituting the expressions for the derivatives in $\dot{V}$, 
it follows from \eqref{eq:model:1} that
\begin{equation}
\label{eq:difV:1}
\begin{split}
\dot{V} 
= &\left(1-\frac{S^*}{S}\right)\left[ \Lambda - \beta_1 \left( 
I + \eta_C \, C  + \eta_A  A \right) S - \mu S \right]\\
&+ \left(1-\frac{I^*}{I}\right)\left[  \beta_1 \left( I 
+ \eta_C \, C  + \eta_A  A \right) S - \xi_3 I + \alpha A + \omega C \right]\\
&+ \frac{\omega}{\xi_2}  \left(1-\frac{C^*}{C}\right)\left[  \phi I -  \xi_2 C \right] 
+ \frac{\alpha}{\xi_1} \left(1-\frac{A^*}{A}\right)\left[  \rho I - \xi_1 A \right].
\end{split}
\end{equation}
Using relation $\Lambda = \beta_1 \left( I^* 
+ \eta_C \, C^*  + \eta_A  A^* \right) S^* + \mu S^*$, 
we have from the first equation of system \eqref{eq:model:2} 
at steady-state that \eqref{eq:difV:1} can be written as
\begin{equation*}
\begin{split}
\dot{V} 
= &\left(1-\frac{S^*}{S}\right)\left[ \beta_1 \left( I^* 
+ \eta_C \, C^*  + \eta_A  A^* \right) S^* + \mu S^* 
- \beta_1 \left( I + \eta_C \, C  + \eta_A  A \right) S - \mu S \right]\\
&+ \left(1-\frac{I^*}{I}\right)\left[  \beta_1 \left( I 
+ \eta_C \, C  + \eta_A  A \right) S - \xi_3 I + \alpha A + \omega C \right]\\
&+ \frac{\omega}{\xi_2}  \left(1-\frac{C^*}{C}\right)\left[  \phi I -  \xi_2 C \right] 
+ \frac{\alpha}{\xi_1} \left(1-\frac{A^*}{A}\right)\left[  \rho I - \xi_1 A \right],
\end{split}
\end{equation*}
which can then be simplified to
\begin{equation*}
\begin{split}
\dot{V} = &\left(1-\frac{S^*}{S}\right) \beta_1 I^* S^* 
+ \mu S^* \left( 2 - \frac{S}{S^*} - \frac{S^*}{S}\right) - \beta_1 I S + \beta_1 I S^*\\
&+ \beta_1 ( \eta_C C^* + \eta_A A^*) S^* - \beta_1 (\eta_C C + \eta_A A) S 
- \frac{S^*}{S} \beta_1 (\eta_C C^* + \eta_A A^*) S^*\\ 
&+ S^* \beta_1 (\eta_C C + \eta_A A)
+ \left(1-\frac{I^*}{I}\right)\left[  \beta_1 \left( I + \eta_C \, C  + \eta_A  A \right) S 
- \xi_3 I + \alpha A + \omega C \right]\\
&+ \frac{\omega}{\xi_2}  \left(1-\frac{C^*}{C}\right)\left[  \phi I -  \xi_2 C \right] 
+ \frac{\alpha}{\xi_1} \left(1-\frac{A^*}{A}\right)\left[  \rho I - \xi_1 A \right].
\end{split}
\end{equation*}
Using the relations at the steady state
\begin{equation*}
\xi_3 I^* = \beta_1 (I^* + \eta_C C^* + \eta_A A^*) S^*  + \alpha A^*  + \omega C^*, 
\quad \xi_2 C^* = \phi I^*, 
\quad \xi_1 A^* = \rho I^*,
\end{equation*} 
and after some simplifications, we have
\begin{equation*}
\begin{split}
\dot{V} = &\left( \beta_1 I^* S^* + \mu S^* \right) \left(2 - \frac{S}{S^*} -\frac{S^*}{S} \right) 
+ \beta_1 S^*\left( \eta_C C^* + \eta_A A^* \right) \left( 1 - \frac{S^*}{S} \right)\\
&+ \beta_1 S^*\left( \eta_C C^* + \eta_A A^* \right) \left( 1 - \frac{I}{I^*} \right)
+ \beta_1 S^* \left( \eta_C C + \eta_A A \right) \left( 1 - \frac{I^*}{I} \frac{S}{S^*}\right)\\   
&+ \alpha A^* \left( 1 - \frac{A}{A^*} \frac{I^*}{I} \right) 
+ \omega C^* \left( 1 - \frac{C}{C^*} \frac{I^*}{I} \right)\\
&+  \frac{\omega \phi}{\xi_2} I^* \left( 1 - \frac{I}{I^*} \frac{C^*}{C} \right)
+ \frac{\alpha \rho}{\xi_1} I^* \left( 1 - \frac{I}{I^*} \frac{A^*}{A} \right).
\end{split}
\end{equation*}
The terms between the larger brackets are less than or equal to zero by a well-known inequality: 
the geometric mean is less than or equal to the arithmetic mean. The equality
$\frac{dV}{dt} = 0$ holds if and only if $(S, I, C, A)$ take the equilibrium 
values $(S^*, I^*, C^*, A^*)$. Therefore, by LaSalle's Invariance Principle \cite{LaSalle1976}, 
the endemic equilibrium $\Sigma_+$ is globally asymptotically stable. 
\end{proof}

We conjecture that the endemic equilibrium for positive AIDS-induced death rate ($d>0$) 
is globally asymptotically stable. This remains an open question.


\subsection{Case study for Cape Verde}
\label{sec:model:CapeVerde}

In this section, we calibrate model \eqref{eq:model:1} to the cumulative cases 
of infection by HIV and AIDS in Cape Verde from 1987 to 2014. We show that 
our model predicts well this reality. In Table~\ref{table:realdataCapeVerde}, 
the cumulative cases of infection by HIV and AIDS in Cape Verde are depicted 
for the years 1987--2014 \cite{report:HIV:AIDS:capevert2015}. 
\begin{table}[!htb]
\centering
\caption{Cumulative cases of infection by HIV/AIDS and total population in Cape Verde 
in the period 1987--2014 \cite{report:HIV:AIDS:capevert2015,WorldBank:TotalPop:url}.}
\label{table:realdataCapeVerde}
\begin{tabular}{l l l l l l l l} \hline  \hline
{\small{Year}} & {\small{1987}} & {\small{1988}} & {\small{1989}} & {\small{1990}} 
& {\small{1991}} & {\small{1992}} & {\small{1993}} \\ \hline
{\small{HIV/AIDS}} & {\small{61}} & {\small{107}}  &  {\small{160}} &  {\small{211}} 
&  {\small{244}} & {\small{303}} &  {\small{337}}\\ 
{\small{Population}} & {\small{323972}} & {\small{328861}}  &  {\small{334473}} &  {\small{341256}} 
&  {\small{349326}} & {\small{358473}} &  {\small{368423}}\\ \hline \hline
{\small{Year}} & {\small{1994}} & {\small{1995}} & {\small{1996}} & {\small{1997}} &  {\small{1998}} 
& {\small{1999}} & {\small{2000}} \\ \hline
{\small{HIV/AIDS}} &  {\small{358}} & {\small{395}}  & {\small{432}}		
& {\small{471}} & {\small{560}} & {\small{660}}  &  {\small{779}} \\
{\small{Population}} &  {\small{378763}} & {\small{389156}}  & {\small{399508}}	& {\small{409805}} 
& {\small{419884}}  &  {\small{429576}} &  {\small{438737}}\\ \hline \hline
{\small{Year}} & {\small{2001}} & {\small{2002}} & {\small{2003}} & {\small{2004}} & {\small{2005}}  
& {\small{2006}} & {\small{2007}} \\ \hline
{\small{HIV/AIDS}} &  {\small{913}} & {\small{1064}} &  {\small{1233}} &  {\small{1493}}  &  {\small{1716}}  
& {\small{2015}} & {\small{2334}} \\
{\small{Population}} & {\small{447357}} & {\small{455396}} &  {\small{462675}} &  {\small{468985}} & {\small{474224}}  
& {\small{478265}}	& {\small{481278}} \\ \hline \hline
{\small{Year}} & {\small{2008}} &  {\small{2009}} & {\small{2010}} & {\small{2011}} 
& {\small{2012}} & {\small{2013}} & {\small{2014}}\\ \hline
{\small{HIV/AIDS}} & {\small{2610}} & {\small{2929}} & {\small{3340}}  
&  {\small{3739}} &  {\small{4090}} & {\small{4537}} & {\small{4946}}\\
{\small{Population}} & {\small{483824}}  &  {\small{486673}} &  {\small{490379}} 
&  {\small{495159}} & {\small{500870}} &  {\small{507258}} & {\small{513906}} \\ \hline\hline
\end{tabular}
\end{table}
We consider the initial conditions \eqref{eq:initcond:CV} based 
on \cite{report:HIV:AIDS:capevert2015,url:worlbank:capevert}: 
\begin{equation}
\label{eq:initcond:CV}
S_0 = S(0) = 323911 \, , 
\quad I_0 = I(0) = 61\, , 
\quad C_0 = C(0) = 0\, , 
\quad A_0 = A(0) = 0.
\end{equation}
We borrow the parameter values $\phi = 1$, $\rho = 0.1$, $\alpha = 0.33$ and $\omega = 0.09$ 
from \cite{SilvaTorres:TBHIV:2015}. Following the World Bank data 
\cite{url:worlbank:capevert,WorldBank:TotalPop:url},
the natural death rate is assumed to take the value $\mu = 1/69.54$. 
The recruitment rate $\Lambda = 13045$ was estimated in order to
approximate the values of the total population of Cape Verde
given in Table~\ref{table:realdataCapeVerde}. See 
Figure~\ref{fig:TotalPop}, were we can observe that model 
\eqref{eq:model:1} fits well the total population of Cape Verde.
\begin{figure}[!htb]
\centering
\includegraphics[width=0.6\textwidth]{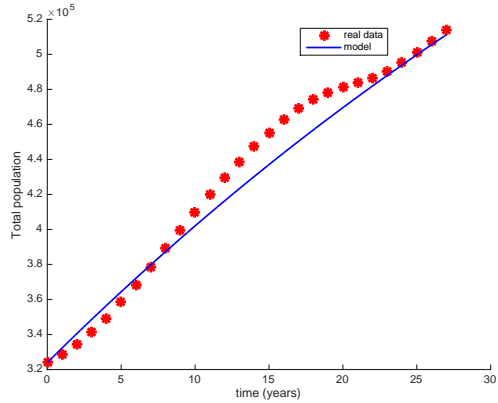}
\caption{Model \eqref{eq:model:1} fitting the total population of Cape Verde 
between 1987 and 2014 \cite{report:HIV:AIDS:capevert2015,WorldBank:TotalPop:url}.
The $l_2$ norm of the difference between the real total population
of Cape Verde and our prediction gives an error of $1.9\%$ of individuals
per year with respect to the total population of Cape Verde in 2014.}
\label{fig:TotalPop}
\end{figure}
The AIDS induced death rate is assume to be $d = 1$ based on \cite{ZwahlenEggerUNAIDS}. 
Two cases are considered: $\eta_C = 0.04$, based on a research study known as \emph{HPTN 052}, 
where it is found that the risk of HIV transmission among heterosexual serodiscordant 
couples is 96\% lower when the HIV-positive partner is on treatment \cite{Cohen:NEJM:2011}; 
and $\eta_C = 0.015$, which means that HIV-infected individuals under ART treatment 
have a very low probability of transmitting HIV, based on \cite{DelRomero:2016}. 
For the modification parameter $\eta_A \geq 1$ that accounts for the relative
infectiousness of individuals with AIDS symptoms, in comparison to those
infected with HIV with no AIDS symptoms, we assume $\eta_A = 1.3$ and $\eta_A = 1.35$, 
based in \cite{art:viral:load}. We estimated the value of the HIV transmission rate 
$\beta$ for $(\eta_C, \eta_A) = (0.04, 1.35)$ equal to $0.695$ and for $(\eta_C, \eta_A) = (0.015, 1.3)$ 
equal to $0.752$, and show that the model \eqref{eq:model:1} predicts well the reality of Cape Verde 
for these parameter values: see Figure~\ref{fig:model:fit}. 
All the considered parameter values are resumed in Table~\ref{table:parameters:HIV:CV}. 
\begin{figure}[!htb]
\centering
\subfloat[\footnotesize{$(\beta, \eta_C, \eta_A) = (0.752, 0.015, 1.3)$}]{\label{model:fit:1}
\includegraphics[width=0.48\textwidth]{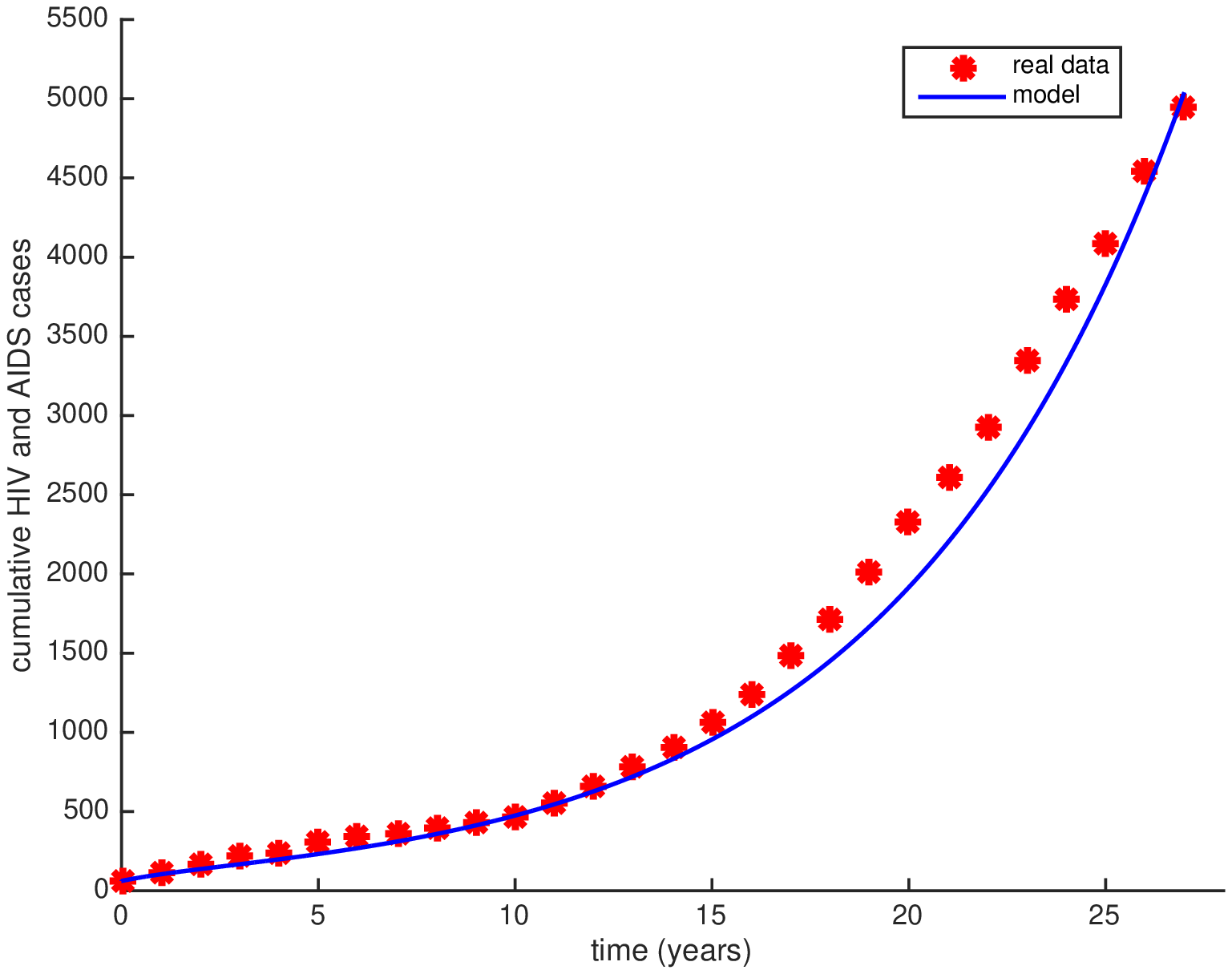}}
\subfloat[\footnotesize{$(\beta, \eta_C, \eta_A) = (0.695, 0.04, 1.35)$}]{\label{model:fit:2}
\includegraphics[width=0.48\textwidth]{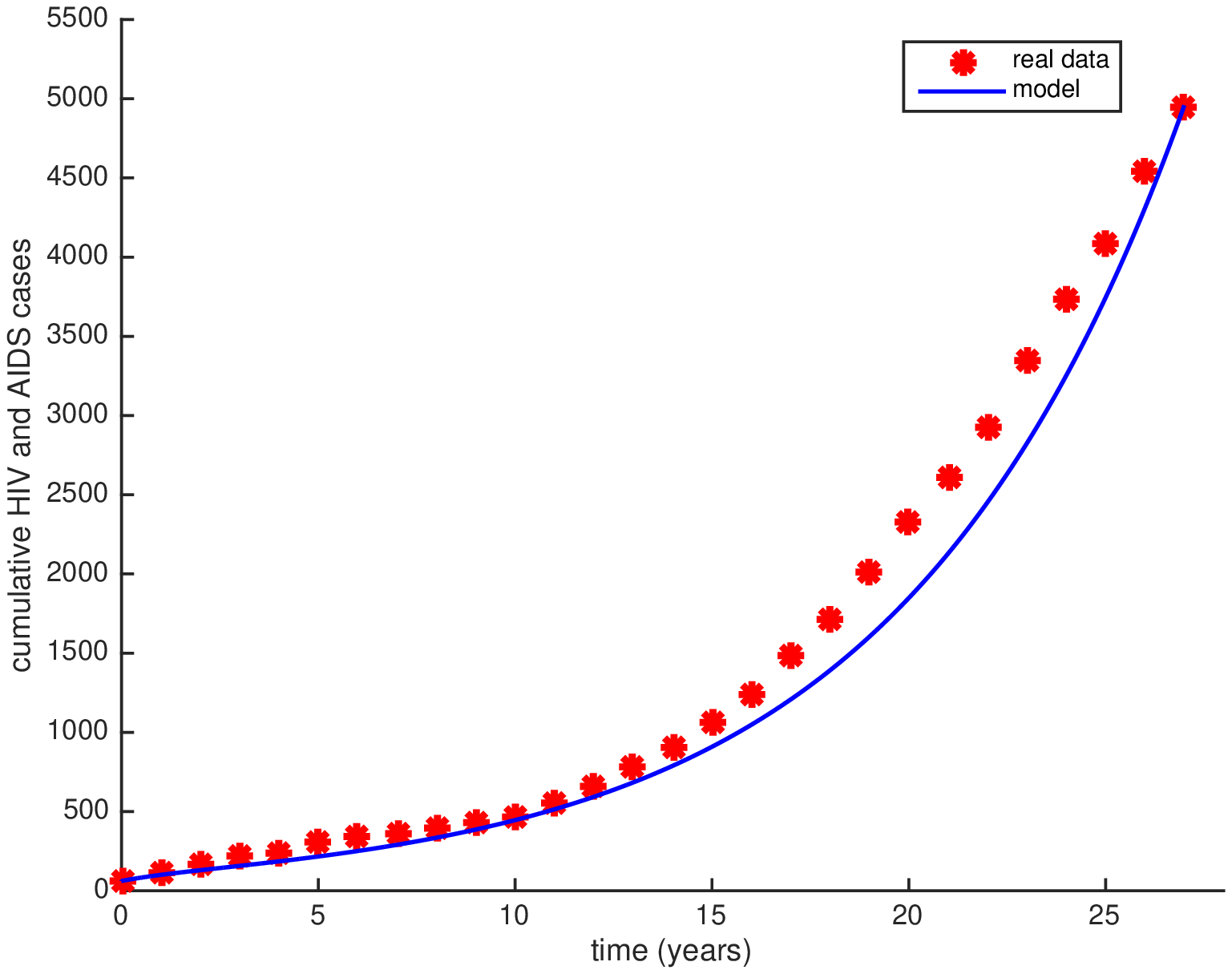}}
\caption{Model \eqref{eq:model:1} fitting the data of cumulative cases of HIV and AIDS infection  
in Cape Verde between 1987 and 2014 \cite{report:HIV:AIDS:capevert2015}. 
The $l_2$ norm of the difference between the real data and the cumulative cases of infection 
by HIV/AIDS given by model \eqref{eq:model:1} gives, in both cases, 
an error of $0.03\%$ of individuals per year with respect 
to the total population of Cape Verde in 2014.}
\label{fig:model:fit}
\end{figure}
\begin{table}[!htb]
\centering
\caption{Parameters of the HIV/AIDS model \eqref{eq:model:1} for Cape Verde.}
\label{table:parameters:HIV:CV}
\begin{tabular}{l  p{6.5cm} l l}
\hline \hline
{\small{Symbol}} &  {\small{Description}} & {\small{Value}} & {\small{Reference}}\\
\hline
{\small{$N(0)$}} & {\small{Initial population}} & {\small{$323 972$}}  
& {\small{\cite{url:worlbank:capevert}}}\\
{\small{$\Lambda$}} & {\small{Recruitment rate}} & {\small{$13045$}}  
& {\small{\cite{url:worlbank:capevert} }}\\
{\small{$\mu$}} & {\small{Natural death rate}} & {\small{$1/69.54$}} 
& {\small{\cite{url:worlbank:capevert} }}\\
{\small{$\beta$}} & {\small{HIV transmission rate}} & {\small{$0.752$}} & {\small{Estimated}}\\
{\small{$\eta_C$}} & {\small{Modification parameter}} & {\small{$0.015$, $0.04$}} & {\small{Assumed}}\\
{\small{$\eta_A$}} & {\small{Modification parameter}} & {\small{$1.3$, $1.35$}} & {\small{Assumed}}\\	
{\small{$\phi$}} & {\small{HIV treatment rate for $I$ individuals}} &  {\small{$1$}} 
& {\small{\cite{SilvaTorres:TBHIV:2015}}} \\
{\small{$\rho$}} & {\small{Default treatment rate for $I$ individuals}}
& {\small{$0.1 $}} & {\small{\cite{SilvaTorres:TBHIV:2015}}}\\
{\small{$\alpha$}} & {\small{AIDS treatment rate}}
& {\small{$0.33 $}} & {\small{\cite{SilvaTorres:TBHIV:2015}}}\\
{\small{$\omega$}} & {\small{Default treatment rate for $C$ individuals}}
& {\small{$0.09$}} & {\small{\cite{SilvaTorres:TBHIV:2015}}}\\
{\small{$d$}} & {\small{AIDS induced death rate}} & {\small{$1$}} 
& {\small{\cite{ZwahlenEggerUNAIDS}}}\\
\hline \hline
\end{tabular}
\end{table}

For the triplets $(\beta, \eta_C, \eta_A) = (0.752, 0.015, 1.3)$ 
and $(\beta, \eta_C, \eta_A) = (0.695, 0.04, 1.35)$, and the other parameter 
values from Table~\ref{table:parameters:HIV:CV}, we have that the basic 
reproduction number is given by $R_0 = 4.0983$ and $R_0 = 4.5304$, respectively. 


\section{The SICAE model} 
\label{sec:SICAE:model}

Now we generalize the model proposed in Section~\ref{sec:SICA:model} 
by adding the possibility of providing PrEP to susceptible individuals.
We add a class of individuals to the total population $N$, denoted by $E$, 
which represents the individuals that are under PrEP. The proportion 
of susceptible individuals that takes PrEP is denoted by $\psi$.
We assume that PrEP is effective so that all susceptible individuals 
under PrEP treatment are transferred to class $E$. The individuals 
that stop PrEP become susceptible individuals again, at a rate $\theta$. 
Individuals under PrEP may suffer of natural death at a rate $\mu$. 
The model is given by the following system of ordinary differential equations:
\begin{equation}
\label{eq:model:PreP}
\begin{cases}
\dot{S}(t) = \Lambda - \frac{\beta \left( I(t) + \eta_C \, C(t)  
+ \eta_A  A(t) \right)}{N(t)} S(t) - \mu S(t) - \psi S(t) + \theta E(t),\\[0.2 cm]
\dot{I}(t) = \frac{\beta \left( I(t) + \eta_C \, C(t)  
+ \eta_A  A(t) \right)}{N(t)} S(t) - (\rho + \phi + \mu) I(t) 
+ \alpha A(t)  + \omega C(t), \\[0.2 cm]
\dot{C}(t) = \phi I(t) - (\omega + \mu)C(t),\\[0.2 cm]
\dot{A}(t) =  \rho \, I(t) - (\alpha + \mu + d) A(t) ,\\[0.2 cm]
\dot{E}(t) = \psi S(t) - (\mu + \theta) E(t).
\end{cases}
\end{equation}
We consider the biologically feasible region
\begin{equation*}
\Omega_P = \left\lbrace (S, I, C, A, E) \in \R_{+0}^5 \, 
: \, S \leq \frac{ \left( \theta+\mu \right) \Lambda}{
\mu\, \left( \theta+\psi+\mu \right) }, 
E \leq \frac{\psi\,\Lambda}{\mu\, \left( \theta+\psi+\mu\right) }, \, 
N \leq \frac{\Lambda}{\mu} \right\rbrace.
\end{equation*}


\subsection{Existence and stability of the disease-free equilibrium}

Model \eqref{eq:model:PreP} has a disease-free equilibrium, given by
\begin{equation}
\label{eq:DFE:Prep}
\Sigma_0 = \left( S^0, I^0, C^0, A^0, E^0  \right) 
= \left( {\frac { \left( \theta+\mu \right) \Lambda}{\mu\, 
\left( \theta+\psi+\mu \right)}} ,0, 0,0, 
{\frac {\psi\,\Lambda}{\mu\, \left( \theta+\psi+\mu\right)}}\right).
\end{equation} 
The linear stability of $\Sigma_0$ can be obtained using the next-generation 
method on system \eqref{eq:model:PreP}. Following \cite{van:den:Driessche:2002}, 
the basic reproduction number for model \eqref{eq:model:PreP} 
is given by \eqref{eq:R0:model:1}, that is, 
\begin{equation*}
R_0 = \frac{ \beta\, \left(  \xi_2  \left( \xi_1 +\rho\, \eta_A \right) 
+ \eta_C \,\phi \, \xi_1 \right) }{\mu\, \left(  \xi_2  \left( \rho + \xi_1
\right) +\phi\, \xi_1 +\rho\,d \right) +\rho\,\omega\,d} 
= \frac{\mathcal{N}}{\mathcal{D}},
\end{equation*}
where $\xi_1 = \alpha + \mu + d$, 
$\xi_2 = \omega + \mu$, and $\xi_3 = \rho + \phi + \mu$.
Thus, from Theorem~2 of \cite{van:den:Driessche:2002},
the following result is established.

\begin{lemma}
\label{lemma:localstab:DFE:Prep}
The disease free equilibrium $\Sigma_0$ of model \eqref{eq:model:PreP}, 
given by \eqref{eq:DFE:Prep}, is locally asymptotically stable 
if $R_0 < 1$, and unstable if $R_0 > 1$. 
\end{lemma}

Biologically speaking, Lemma~\ref{lemma:localstab:DFE:Prep} implies 
that HIV infection can be eliminated from the community (when $R_0 < 1$) 
if the initial size of the population is in the basin of attraction 
of $\Sigma_0$. To ensure that elimination of HIV infection 
is independent of the initial size of the population, 
it is necessary to show that the disease free equilibrium 
is globally asymptotically stable \cite{Liu:Zhang:MCM:2011}. 
This is obtained in what follows. 

\begin{theorem}
\label{thm:gasS0}
The disease free equilibrium $\Sigma_0$ 
is globally asymptotically stable for $R_0 < 1$. 
\end{theorem}

\begin{proof}
Consider the following Lyapunov function:
\begin{equation*}
\begin{split}
V = &\left( \xi_1 \xi_2 + \xi_1 \phi \eta_C + \xi_2 \rho \eta_A \right) I 
+ \left( \xi_1 \omega + \xi_1 \xi_3 \eta_C + \rho \eta_A \omega - \eta_C \rho \alpha \right) C\\ 
&+ \left( \alpha \xi_2 + \xi_2 \xi_3 \eta_A + \phi \eta_C \alpha - \phi \eta_A \omega \right) A.
\end{split}
\end{equation*}
The time derivative of $V$ computed along the solutions of \eqref{eq:model:1} is given by
\begin{equation*}
\begin{split}
\dot{V} 
&= \left( \xi_1 \xi_2 + \xi_1 \phi \eta_C + \xi_2 \rho \eta_A \right) \dot{I} 
+ \left( \xi_1 \omega + \xi_1 \xi_3 \eta_C + \rho \eta_A \omega 
- \eta_C \rho \alpha \right) \dot{C}\\ 
& \quad + \left( \alpha \xi_2 + \xi_2 \xi_3 \eta_A 
+ \phi \eta_C \alpha - \phi \eta_A \omega \right) \dot{A}\\
&= \left( \xi_1 \xi_2 + \xi_1 \phi \eta_C + \xi_2 \rho \eta_A \right) 
\left( \frac{\beta}{N} \left( I + \eta_C \, C  + \eta_A  A \right) S 
- \xi_3 I + \alpha A + \omega C \right)\\
& \quad + \left( \xi_1 \omega + \xi_1 \xi_3 \eta_C + \rho \eta_A \omega 
- \eta_C \rho \alpha \right) \left( \phi I - \xi_2 C \right)\\ 
& \quad  + \left( \alpha \xi_2 + \xi_2 \xi_3 \eta_A + \phi \eta_C \alpha 
- \phi \eta_A \omega \right) \left( \rho \, I - \xi_1 A \right).
\end{split}
\end{equation*}	
After some simplifications, we have 
\begin{equation*}
\begin{split}
\dot{V} &= (\xi_1 \xi_2 +  \xi_1 \phi \eta_C  + \xi_2 \rho \eta_A) 
\frac{\beta I S}{N} + (- \xi_1 \xi_2 \xi_3 + \xi_1 \omega \phi + \alpha \xi_2 \rho ) I\\
& \quad + \eta_C (\xi_1 \xi_2 + \xi_1 \phi \eta_C + \xi_2 \rho \eta_A ) 
\frac{\beta C S}{N} + \eta_C (- \xi_1 \xi_3 \xi_2 + \xi_1 \phi \omega +  \rho \alpha \xi_2) C\\
& \quad + \eta_A (\xi_1 \xi_2 + \xi_1 \phi \eta_C + \xi_2 \rho \eta_A ) 
\frac{\beta A S}{N} + \eta_A (- \xi_2 \xi_3 \xi_1 + \phi \omega \xi_1 + \xi_2 \rho \alpha ) A\\
&= \mathcal{D}\left(R_0 \frac{S}{N} -1\right) I + \eta_C \mathcal{D}\left(R_0 \frac{S}{N} -1\right) C 
+ \eta_A \mathcal{D}\left(R_0 \frac{S}{N} -1\right) A\\
&\leq \mathcal{D}(R_0 -1) I + \eta_C \mathcal{D}(R_0 -1) C 
+ \eta_A \mathcal{D}(R_0 -1) A \quad (\text{because} \, S \leq N \, \text{ in } \, \Omega)\\
& \leq 0 \quad \text{for} \,\,  R_0 < 1.
\end{split}
\end{equation*}
Because all model parameters are nonnegative, it follows that $\dot{V} \leq 0$ 
for $R_0 < 1$ with $\dot{V} = 0$ if, and only if, $I=C=A=0$. Substituting 
$(I, C, A ) = (0, 0, 0)$ into the equations for $S$ and $E$ in system \eqref{eq:model:PreP} 
shows, respectively, that $S \to S^0$ and $E \to E^0$ as $t \to \infty$. Thus, 
it follows from LaSalle's Invariance Principle \cite{LaSalle1976} 
that every solution of system \eqref{eq:model:PreP} with initial conditions in $\Omega$ 
approaches the disease free equilibrium $\Sigma_0$ as $t \to \infty$ whenever $R_0 < 1$. 
\end{proof}

The epidemiological significance of Theorem~\ref{thm:gasS0} is that 
HIV infection will be eliminated from the population if the threshold quantity, 
$R_0$, can be brought to a value less than unity.  


\subsection{Existence and stability of the endemic equilibrium}

The unique endemic equilibrium of model \eqref{eq:model:PreP} 
exists whenever $R_0 > 1$ and is given by
\begin{equation*}
\begin{split}
S^* &= \frac{ \Lambda \, \xi_4 \left(\xi_1 (\phi + \xi_2) + \rho \xi_2 \right)}{\mathcal{F} }, 
\quad \quad 
I^* = \frac{- \xi_4 \, \Lambda \xi_1 \xi_2 (\mathcal{D} -\mathcal{N})}{\mathcal{D} \, \mathcal{F}},\\
C^* &= \frac{- \xi_4 \, \Lambda \phi \xi_1 (\mathcal{D} -\mathcal{N})}{\mathcal{D} \, \mathcal{F}},\\
A^* &= \frac{- \xi_4 \, \Lambda \rho \xi_2 (\mathcal{D} -\mathcal{N})}{\mathcal{D} \, \mathcal{F}},
\quad \quad
E^* = \frac{\psi \Lambda \left(\xi_1 (\phi + \xi_2) + \rho \xi_2 \right)}{\mathcal{F}},
\end{split}
\end{equation*}
where $\mathcal{F} = (\mathcal{N} - \rho d \xi_2)\theta + (\mathcal{D} 
- \rho d \xi_2) \psi - \mu( \rho d \xi_2 -\mathcal{N})$ and $\xi_4 = \theta + \mu$. 
We investigate the global stability of the endemic equilibrium 
of model \eqref{eq:model:PreP} for the case when the associated 
AIDS-induced mortality is negligible ($d=0$) and there is a 
strict adherence to PrEP, that is, $\theta = 0$. Adding the equations 
of model \eqref{eq:model:PreP} with $d=0$ and $\theta = 0$ 
gives $\dot{N} = \Lambda - \mu N$, so that $N \to \frac{\Lambda}{\mu}$ 
as $t \to \infty$. Thus, $\frac{\Lambda}{\mu}$ is an upper bound of $N(t)$, 
provided $N(0) \leq \frac{\Lambda}{\mu}$. Further, if $N(0) > \frac{\Lambda}{\mu}$, 
then $N(t)$ decreases to this level. Using $N = \frac{\Lambda}{\mu}$ in the force of infection 
$\lambda = \frac{\beta}{N} \left( I + \eta_C \, C  + \eta_A  A \right)$ 
gives a limiting (mass action) system. Then, the force of infection becomes
$$
\lambda = \beta_1 \left( I + \eta_C \, C  + \eta_A  A \right), 
\quad \text{where} \quad \beta_1 = \frac{\beta \mu}{\Lambda}.
$$
Therefore, we consider the following model:
\begin{equation}
\label{eq:model:PreP:2}
\begin{cases}
\dot{S}(t) = \Lambda - \beta_1 \left( I(t) + \eta_C \, C(t)  
+ \eta_A  A(t) \right) S(t) - (\mu + \psi) S(t),\\[0.2 cm]
\dot{I}(t) = \beta_1 \left( I(t) + \eta_C \, C(t)  
+ \eta_A  A(t) \right) S(t) - \xi_3 I(t) + \alpha A(t) + \omega C(t), \\[0.2 cm]
\dot{C}(t) = \phi I(t) - \xi_2 C(t),\\[0.2 cm]
\dot{A}(t) =  \rho \, I(t) - \xi_1 A(t) ,\\[0.2 cm]
\dot{E}(t) = \psi S(t) - \mu E(t).
\end{cases}
\end{equation}
For system \eqref{eq:model:PreP:2}, the basic reproduction number 
is given by $R_0 = \frac{\Lambda \mathcal{N}_1}{(\mu + \psi) \mathcal{D}}$ 
with $\mathcal{N}_1 = \beta_1 \, \left(  \xi_2  \left( \xi_1 +\rho\, 
\eta_A \right) + \eta_C \,\phi \, \xi_1 \right)$. When $R_0 > 1$, 
system \eqref{eq:model:PreP:2} has a unique endemic equilibrium 
$\tilde{\Sigma}_+ = (\tilde{S}, \tilde{I}, \tilde{C}, \tilde{A}, \tilde{E})$ given by
\begin{equation}
\label{eq:EE:Prep:2}
\begin{split}
\tilde{S} &= \frac{ \mu \left(\xi_1 (\phi + \xi_2) 
+ \rho \xi_2 \right)}{{ \beta_1 (\xi_1(\xi_2 + \eta_C \phi)+ \eta_A \rho \xi_2)} }, 
\quad 
\tilde{I} = \frac{\xi_1 \xi_2 (\mathcal{N} -\mathcal{D})}{\mathcal{D}_1 },\\ 
\tilde{C} &= \frac{\phi \xi_1 (\mathcal{N}-\mathcal{D})}{\mathcal{D}_1},
\quad  
\tilde{A} = \frac{\rho \xi_2 (\mathcal{N} -\mathcal{D})}{\mathcal{D}_1},
\quad 
\tilde{E} = \frac{\psi \left(\xi_1 (\phi + \xi_2) 
+ \rho \xi_2 \right)}{\beta_1 (\xi_1(\xi_2 + \eta_C \phi)+\eta_A \rho \xi_2)},
\end{split}
\end{equation}
where $\mathcal{D}_1 = \beta_1 \mu (\rho \xi_2 
+ \xi_1 (\phi + \xi_2))(\xi_1(\xi_2 + \eta_C \phi)+ \eta_A \rho \xi_2)$. 
Since $E$ does not appear in the first four equations 
of system \eqref{eq:model:PreP:2}, 
we only need to consider the reduced system
\begin{equation}
\label{eq:model:PreP:reduced}
\begin{cases}
\dot{S}(t) = \Lambda - \beta_1 \left( I(t) + \eta_C \, C(t)  
+ \eta_A  A(t) \right) S(t) - (\mu + \psi) S(t),\\[0.2 cm]
\dot{I}(t) = \beta_1 \left( I(t) + \eta_C \, C(t)  + \eta_A  A(t) \right) S(t) 
- \xi_3 I(t) + \alpha A(t)  + \omega C(t), \\[0.2 cm]
\dot{C}(t) = \phi I(t) - \xi_2 C(t),\\[0.2 cm]
\dot{A}(t) =  \rho \, I(t) - \xi_1 A(t).
\end{cases}
\end{equation}
Whenever $R_0 > 1$, system \eqref{eq:model:PreP:reduced} has a unique 
endemic equilibrium $\Sigma_1 = (\tilde{S}, \tilde{I}, \tilde{C}, \tilde{A})$ 
with $\tilde{S}$, $\tilde{I}$, $\tilde{C}$ and $\tilde{A}$ 
defined by \eqref{eq:EE:Prep:2}. 
Consider the Lyapunov function \eqref{eq:Lyapunov:function} given by
\begin{equation*}
V = \left(S - S^* \ln(S)\right) + \left(I - I^* \ln(I)\right) 
+ \frac{\omega}{\xi_2} \left(C - C^* \ln(C)\right) 
+ \frac{\alpha}{\xi_1} \left(A - A^* \ln(A)\right).
\end{equation*}
Differentiating $V$ with respect to time gives
\begin{equation*}
\dot{V} = \left(1-\frac{S^*}{S}\right)\dot{S} + \left(1-\frac{I^*}{I}\right)\dot{I} 
+ \frac{\omega}{\xi_2} \left(1-\frac{C^*}{C}\right)\dot{C} + \frac{\alpha}{\xi_1} 
\left(1-\frac{A^*}{A}\right)\dot{A}.
\end{equation*}
Substituting the expressions for the derivatives in $\dot{V}$, 
it follows from \eqref{eq:model:PreP:2} that
\begin{equation}
\label{eq:difV:2}
\begin{split}
\dot{V} 
= &\left(1-\frac{S^*}{S}\right)\left[ \Lambda - \beta_1 \left( I 
+ \eta_C \, C  + \eta_A  A \right) S - \mu S - \psi S \right]\\
&+ \left(1-\frac{I^*}{I}\right)\left[  \beta_1 \left( I + \eta_C \, C 
+ \eta_A  A \right) S - \xi_3 I + \alpha A + \omega C \right]\\
&+ \frac{\omega}{\xi_2}  \left(1-\frac{C^*}{C}\right)\left[  \phi I -  \xi_2 C \right] 
+ \frac{\alpha}{\xi_1} \left(1-\frac{A^*}{A}\right)\left[\rho I - \xi_1 A \right].
\end{split}
\end{equation}
Using relation $\Lambda = \beta_1 \left( I^* + \eta_C \, C^*  
+ \eta_A  A^* \right) S^* + \mu S^* + \psi S^*$, we have from the 
first equation of system \eqref{eq:model:PreP:2} at steady-state 
that \eqref{eq:difV:2} can be written as
\begin{equation*}
\begin{split}
\dot{V} 
= \left(1-\frac{S^*}{S}\right) &\biggl[ 
\beta_1 \left( I^* + \eta_C \, C^*  + \eta_A  A^* \right) S^* + \mu S^* + \psi S^*\\ 
&- \beta_1 \left( I + \eta_C \, C  + \eta_A  A \right) S - \mu S - \psi S \biggr]\\
&+ \left(1-\frac{I^*}{I}\right)\left[  \beta_1 \left( I + \eta_C \, C  + \eta_A  A \right) S 
- \xi_3 I + \alpha A + \omega C \right]\\
&+ \frac{\omega}{\xi_2}  \left(1-\frac{C^*}{C}\right)\left[  \phi I -  \xi_2 C \right] 
+ \frac{\alpha}{\xi_1} \left(1-\frac{A^*}{A}\right)\left[  \rho I - \xi_1 A \right],
\end{split}
\end{equation*}
which can be simplified to
\begin{equation*}
\begin{split}
\dot{V} = &\left(1-\frac{S^*}{S}\right) \beta_1 I^* S^* 
+ \mu S^* \left( 2 - \frac{S}{S^*} - \frac{S^*}{S}\right) 
+ \psi S^* \left( 2 - \frac{S}{S^*} - \frac{S^*}{S}\right) + \beta_1 I S^*\\
&+ \beta_1 S^*  ( \eta_C C^* + \eta_A A^*) \left(1 - \frac{S^*}{S}\right) 
+ S^* \beta_1 (\eta_C C + \eta_A A)
- \xi_3 I + \frac{\omega}{\xi_2} \phi I + \frac{\alpha}{\xi_1} \rho I \\
&-\frac{I^*}{I} \left[  \beta_1 \left( I + \eta_C \, C  
+ \eta_A  A \right) S - \xi_3 I + \alpha A + \omega C \right]\\
&- \frac{\omega}{\xi_2}  \frac{C^*}{C} \left[  \phi I -  \xi_2 C \right]
- \frac{\alpha}{\xi_1} \frac{A^*}{A} \left[  \rho I - \xi_1 A \right].
\end{split}
\end{equation*}
Using the relations 
\begin{equation*}
\xi_3 I^* = \beta_1 (I^* + \eta_C C^* + \eta_A A^*) S^*  + \alpha A^*  + \omega C^*, 
\quad \xi_2 C^* = \phi I^*, 
\quad \xi_1 A^* = \rho I^*
\end{equation*} 
at the steady state, after some simplifications we have
\begin{equation*}
\begin{split}
\dot{V} = &\beta_1 I^* S^* \left(1-\frac{S^*}{S}\right)  
+ \mu S^* \left( 2 - \frac{S}{S^*} - \frac{S^*}{S}\right) 
+ \psi S^* \left( 2 - \frac{S}{S^*} - \frac{S^*}{S}\right)\\
&+ \beta_1 S^*  (\eta_C C^* + \eta_A A^*) \left(1 - \frac{S^*}{S} \right) 
+ \beta_1 S^*  ( \eta_C C^* + \eta_A A^*) \left(1 - \frac{I}{I^*}\right)\\  
&+ S^* \beta_1 (\eta_C C + \eta_A A) \left( 1 - \frac{I^*}{I} \frac{S}{S^*}\right) \\
&+ \alpha A^* \left( 1 - \frac{A}{A^*} \frac{I^*}{I} \right) 
+ \omega C^* \left( 1 - \frac{C}{C^*} \frac{I^*}{I} \right)\\
&+ \frac{\omega \phi}{\xi_2} I^* \left( 1 - \frac{I}{I^*} \frac{C^*}{C} \right)
+ \frac{\alpha \rho}{\xi_1} I^* \left( 1 - \frac{I}{I^*} \frac{A^*}{A} \right).
\end{split}
\end{equation*}
The terms between the larger brackets are less than or equal to zero 
because the geometric mean is less than or equal to the arithmetic mean. 
Equality $\frac{dV}{dt} = 0$  holds if and only if $S=S^*$, $I=I^*$, $C=C^*$ and $A=A^*$. 
By LaSalle's Invariance Principle \cite{LaSalle1976}, every solution 
to the equations in model \eqref{eq:model:PreP:reduced} 
with initial conditions in $\left\lbrace (S, I, C, A) \in \R_{+0}^4 \, : \, S 
\leq \frac{\Lambda}{\psi+\mu }, N \leq \frac{\Lambda}{\mu} \right\rbrace \backslash  
\left\lbrace S = I = C = A \right\rbrace$ approaches the endemic equilibrium $\tilde{\Sigma}_+$, 
which implies that the endemic equilibrium $\tilde{\Sigma}_+$ of system \eqref{eq:EE:Prep:2} 
is globally asymptotically stable on $\Omega_P \backslash \Omega_{P0}$, 
where $\Omega_{P0} = \{ \left( S, I, C, A \right) \in \Omega_P \, : \, I=C=A=0 \}$. 
We have just proved the following result. 

\begin{theorem}
\label{thm:need:gen:op}
The endemic equilibrium $\tilde{\Sigma}_+$ of model \eqref{eq:model:PreP:2} 
is globally asymptotically stable in $\Omega_P \backslash \Omega_{P0}$ whenever $R_0 > 1$.  
\end{theorem}

A generalization of Theorem~\ref{thm:need:gen:op} to the case 
$d > 0$ and $\theta > 0$ remains an open question. 


\subsection{Numerical simulations}
\label{sec:numSimu:SICAE}

In this section, we investigate the impact 
of PrEP in the reduction of HIV transmission. 
We assume that the total population is constant, that is, 
$\Lambda = \mu N$ with $\mu = 1/69.54$ and $d = 0$. 
The initial conditions are given by
\begin{equation}
\label{eq:init:SICAE}
S(0) = 10000, \, I(0) = 200, \, C(0) = 0, \,  
A(0) = 0 \quad \text{and} \quad E(0) = 0.
\end{equation}
We start our numerical simulations assuming that 10 per cent of the susceptible individuals take PrEP, 
that is, $\psi = 0.1$ and the default rate takes the value $\theta = 0.001$. Consider $\beta = 0.582$, 
$\eta_C = 0.04$, $\eta_A = 1.35$, and that the parameters $\omega$, $\rho$, $\phi$ and $\alpha$ 
take the values of Table~\ref{table:parameters:HIV:CV}. We compare the case $(\psi, \theta) 
= (0.1, 0.001)$ with $(\psi, \theta) = (0, 0)$ for $t \in [0, t_f]$, $t_f = 25$ years, 
where  $(\psi, \theta) = (0, 0)$ means that no susceptible individual was under PrEP. 
\begin{figure}[!htb]
\centering
\includegraphics[width=1.00\textwidth]{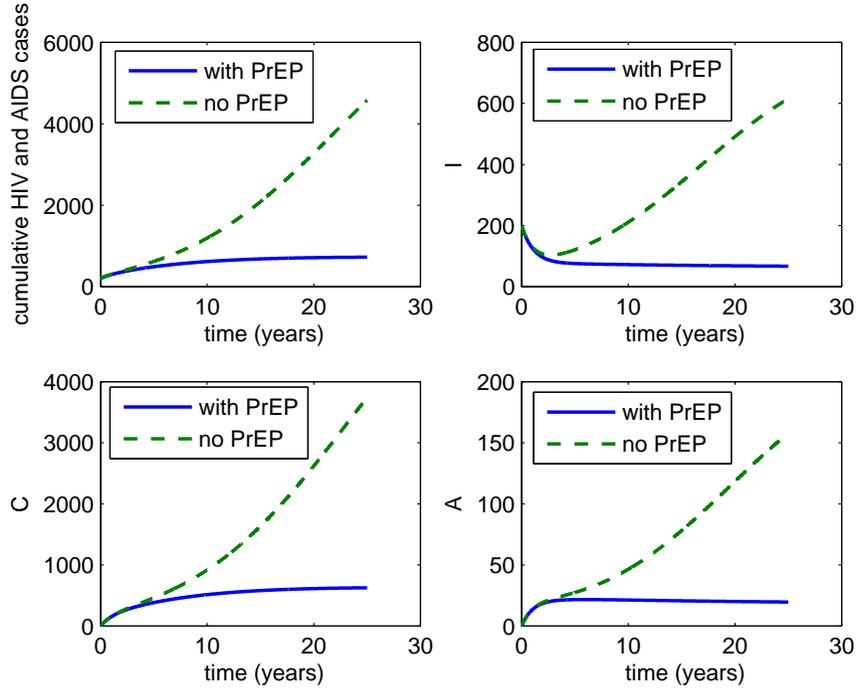}
\caption{Top left: cumulative HIV and AIDS cases. Top right: 
pre-AIDS HIV infected individuals $I$. Bottom left: HIV-infected 
individuals under ART treatment $C$. Bottom right: HIV-infected 
individuals with AIDS symptoms $A$. Expression ``\emph{with PrEP}'' 
refers to the case $(\psi, \theta) = (0.1, 0.001)$ and 
``\emph{no PrEP}'' refers to the case $(\psi, \theta) = (0, 0)$.}
\label{fig:SICAE:PrEP:YesNo}
\end{figure}
From Figure~\ref{fig:SICAE:PrEP:YesNo}, we observe that PrEP reduces 
the number of individuals with HIV infection. In fact, 
for $(\psi, \theta) = (0.1, 0.001)$, we have $S(25) \simeq 1687$, 
$I(25) \simeq 67$, $C(25) \simeq 626$, $A(25) \simeq 20$ and $E(25) = 7800$. 
It is important to note that $I(t)$ is a decreasing function 
for all $t \in [0, 25]$ and that the maximum value for the number 
of HIV-infected individuals with AIDS symptoms is less than 22 
for $t \in [0, 25]$. On the other hand, at the end of the 25 years, 
the number of individuals that are taking PrEP is equal to 7800, 
which is highly expensive (the PrEP drug costs between \$8,000 
and \$14,000 per year for each individual). Therefore,
it is of most importance to establish what is the optimal proportion 
of susceptible individuals that should take PrEP, taking into consideration 
its costs. In Section~\ref{sec:optimal:control}, we formulate 
this problem mathematically, using the theory of optimal control, 
and we study it both analytically and numerically.  


\section{Optimal control problem with a mixed state control constraint}
\label{sec:optimal:control}

Substantial gaps remain in understanding the trade-offs between 
costs and benefits of choosing alternative HIV prevention strategies, 
such as the initiation of PrEP by high risk uninfected individuals 
\cite{Drabo:PreP:CID:2016}. Following WHO, making PrEP drugs available 
for safe, effective prevention outside the clinical trial setting is the current challenge. 
However, it is important to highlight and recall that PrEP is not for everyone: 
only people who are HIV-negative and at very high risk for HIV infection 
should take PrEP \cite{url:aids:PrEP}. Moreover, PrEP is highly expensive 
and it is still not approved in many countries like, for example, 
by the European Medicines Agency (EMEA) \cite{Review:PrEP:Infection:2016}. Therefore, 
the number of individuals that should take PrEP is limited at each instant 
of time for a fixed interval of time $[0, t_f]$. In order to study this health public problem, 
from an optimal point of view, we formulate an optimal control problem 
with a mixed state control constraint. For the usefulness of such problems
in epidemiology see, e.g., \cite{SEIR:Rosario:2014}. We consider the model with PrEP 
\eqref{eq:model:PreP} and formulate an optimal control problem with the aim to determine 
the PrEP strategy $\psi$ that minimizes the number of individuals with pre-AIDS 
HIV-infection $I$ as well as the costs associated with PrEP. We assume that the fraction 
of individuals that takes PrEP, at each instant of time, is a control function, 
that is, $\psi \equiv u(t)$ with $t \in [0, t_f]$, and that the total population 
$N$ is constant: the recruitment rate is proportional to the natural death rate, 
$\Lambda = \mu N$, and there are no AIDS-induced deaths ($d=0$). Precisely, 
we consider the model with control $u(t)$ given by
\begin{equation}
\label{eq:model:PreP:control}
\begin{cases}
\dot{S}(t) = \mu N - \frac{\beta}{N} \left( I(t) + \eta_C \, C(t)  
+ \eta_A  A(t) \right) S(t) - \mu S(t) - S(t) u(t) + \theta E(t),\\[0.2 cm]
\dot{I}(t) = \frac{\beta}{N} \left( I(t) + \eta_C \, C(t)  
+ \eta_A  A(t) \right) S(t) - (\rho + \phi + \mu)I(t) + \alpha A(t)  + \omega C(t), \\[0.2 cm]
\dot{C}(t) = \phi I(t) - (\omega + \mu)C(t),\\[0.2 cm]
\dot{A}(t) =  \rho \, I(t) - (\alpha + \mu) A(t) ,\\[0.2 cm]
\dot{E}(t) = S(t) u(t) - (\mu + \theta) E(t)
\end{cases}
\end{equation}
and formulate an optimal control problem with the aim to determine the PrEP 
strategy $u$ over a fixed interval of time $[0, t_f]$ that minimizes the cost functional
\begin{equation}
\label{eq:cost}
J(u) = \int_0^{t_f} \left[ w_1 I(t) + w_2 u^2(t) \right]  \, dt,
\end{equation}
where the constants $w_1$ and $w_2$ represent the weights associated with 
the number of HIV infected individuals $I$ and on the cost associated with 
the PrEP prevention treatment, respectively. It is assumed that the control 
function $u$ takes values between 0 and 1. When $u(t) = 0$, no susceptible 
individual takes PrEP at time $t$; if $u(t)=1$, then all susceptible individuals 
are taking PrEP at time $t$. Let $\vartheta$ denote the total number of susceptible 
individuals under PrEP for a fixed time interval $[0, t_f]$. This constraint is represented by
\begin{equation}
\label{eq:constraint}
S(t) u(t) \leq \vartheta \, , \quad \vartheta \geq 0 \, , \, \,
\text{for almost all} \, \, t \in [0, t_f] \, ,
\end{equation}
which should be satisfied at almost every instant of time during the whole PrEP program.
Let 
\begin{equation*}
x(t) =(x_1(t), \ldots,  x_5(t))
=\left( S(t), I(t), C(t), A(t), E(t) \right) 
\in {\mathbb{R}}^5.
\end{equation*}
The optimal control problem consists to find the optimal trajectory $\tilde{x}$, 
associated with the control $\tilde{u}$, satisfying the control system 
\eqref{eq:model:PreP:control}, the initial conditions \eqref{eq:init:SICAE},
\begin{equation*}
x(0) = (x_{10}, x_{20}, x_{30}, x_{40}, x_{50}), 
\quad \text{with} \quad x_{10} \geq 0,\,  x_{20} \geq 0, \, x_{30} \geq 0, \,  
x_{40} \geq 0, \,  x_{50} \geq 0,
\end{equation*}
the constraint \eqref{eq:constraint}, and where the control 
$\tilde{u} \in \Omega$ minimizes the objective functional
\eqref{eq:cost} with
\begin{equation}
\label{eq:admiss:control}
\Omega = \biggl\{ u(\cdot) \in L^{\infty}(0, t_f) \,
| \,  0 \leq u (t) \leq 1  \biggr\}.
\end{equation}
The control system can be rewritten in the following way:
\begin{equation*}
\frac{dx(t)}{dt} = f(x(t)) + A x(t)+ B x(t) u(t)
\end{equation*}
with
\begin{equation*}
A = \begin{pmatrix}
- \mu & 0 & 0 & 0 &0 \\
0 & - \rho - \phi - \mu & \omega & \alpha &0 \\
0 & \phi & - \omega - \mu & 0 &0  \\
0 & \rho & 0 & - \alpha - \mu &0   \\
0 & 0 & 0 & 0 & - \mu - \theta \\
\end{pmatrix}
\end{equation*}
and
\begin{equation*}
B = \left(b \, Z  \right),
\end{equation*}
where $b = \left(-1 \, 0 \, 0\, 0\, 1 \right)^T$ 
and $Z = 0$ with $0$ the $5 \times 4$ null matrix and
$f =  \left(f_1 \, f_2 \, 0\, 0\, 0   \right)$ with
\begin{equation*}
f_1 = \mu N - \frac{\beta}{N} \left( I(t) + \eta_C \, C(t)  + \eta_A  A(t) \right) S(t)
\end{equation*}
and
\begin{equation*}
f_2 = \frac{\beta}{N} \left( I(t) + \eta_C \, C(t)  + \eta_A  A(t) \right) S(t).
\end{equation*}
It follows from Theorem~23.11 in \cite{livro:Clarke:2013} that
problem \eqref{eq:model:PreP:control}--\eqref{eq:admiss:control}
has a solution (see also \cite{SEIR:Rosario:2014}).
Let $(\tilde{x}, \tilde{u})$ denote such solution.
To determine it, we apply the Pontryagin Maximum Principle
(see, e.g., Theorem~7.1 in \cite{SIAM:Clarke:Rosario}): there exist
multipliers $\lambda_0 \leq 0$, $\lambda \in AC([0, t_f]; {\mathbb{R}}^5)$,
and $\nu \in L^1([0, t_f]; {\mathbb{R}})$, such that
\begin{itemize}
\item $\min \{ | \lambda(t) | \, : \, t \in [0, t_f] \}
> \lambda_0$ (nontriviality condition);
	
\item $\frac{d \lambda (t)}{dt}
= - \frac{\partial \mathcal{H}}{\partial x}(\tilde{x}(t), 
\tilde{u}(t), \lambda_0, \lambda(t), \nu(t))$ (adjoint system);
	
\item $\lambda(t) B \tilde{x}(t) + \nu(t) \tilde{x}_1(t)
+ \lambda_0 w_2 \tilde{u}^2(t) \in \mathcal{N}_{[0,1]}(\tilde{u}(t))$ a.e. and
$$
\mathcal{H}(\tilde{x}(t), \tilde{u}(t), \lambda_0, \lambda(t), \nu(t))
\leq \mathcal{H}(\tilde{x}(t), v, \lambda_0, \lambda(t), \nu(t))\, ,
\forall v \in [0, 1] \, :\,  \tilde{x}_1(t)v \leq \vartheta
$$
(minimality condition);
	
\item $\nu(t) (\tilde{x}_1(t) \tilde{u}(t) - \vartheta) = 0$
and $\nu(t) \leq 0$ a.e.
	
\item $\lambda(t_f) = (0, \ldots, 0)$ (transversality conditions);
\end{itemize}
where the Hamiltonian $\mathcal{H}$ for problem
\eqref{eq:model:PreP:control}--\eqref{eq:admiss:control} is defined by
\begin{equation*}
\mathcal{H}(x, u, \lambda_0, \lambda, \nu)
= \lambda_0 \left( w_1 x_2 + w_2 u^2 \right)
+ \lambda \left( f(x) + A x + B x u \right) + \nu (S u - \vartheta)
\end{equation*}
and
$\mathcal{N}_{[0,1]}(\tilde{u}(t))$ stands for the normal cone from convex
analysis to $[0, 1]$ at the optimal control $\tilde{u}(t)$
(see, e.g., \cite{livro:Clarke:2013}). The optimal solution
$(\tilde{x}, \tilde{u})$ is normal (see \cite{SEIR:Rosario:2014} for details),
so we can choose $\lambda_0 = 1$. The unique optimal control $\tilde{u}$ is given by
\begin{equation*}
\tilde{u}(t) = \min \left\{1, \max \left\{0, \frac{1}{2}\frac{\left(\tilde{\lambda}_1(t)
- \tilde{\lambda}_5(t) - \nu(t) \right)\tilde{x}_1(t)}{\omega_2} \right\} \right\},
\end{equation*}
where the adjoint functions satisfy
\begin{equation*}
\begin{cases}
\displaystyle \dot{\tilde{\lambda}}_1 
= \tilde{\lambda}_1 \, \left( \frac{\beta}{N}\, \left( \tilde{x}_2 + \eta_C \, \tilde{x}_3
+\eta_A \, \tilde{x}_4 \right)  + \mu + \tilde{u} \right)\\ 
\qquad \quad - \tilde{\lambda}_2 \frac {\beta }{N} 
\left( \tilde{x}_2 + \eta_C \, \tilde{x}_3 + \eta_A \, \tilde{x}_4 \right)
- \tilde{\lambda}_5 \,\tilde{u}-\nu\,\tilde{u} \\[0.2 cm]
\dot{\tilde{\lambda}}_2 = -\omega_1 + \tilde{\lambda}_1 
\frac {\beta}{N}\, \tilde{x}_1- \tilde{\lambda}_2 \,
\left( \frac {\beta}{N}\, \tilde{x}_1 -\rho-\phi-\mu \right) - \tilde{\lambda}_3 \,
\phi- \tilde{\lambda}_4 \,\rho  \\[0.2cm]
\displaystyle \dot{\tilde{\lambda}}_3 
= \tilde{\lambda}_1 \, \frac {\beta}{N} \eta_C \,\tilde{x}_1
- \tilde{\lambda}_2 \, \left( \frac {\beta}{N} \eta_C \,\tilde{x}_1 
+\omega \right) - \tilde{\lambda}_3 \, \left( -\omega-\mu \right)\\[0.2cm]
\displaystyle \dot{\tilde{\lambda}}_4 
=   \tilde{\lambda}_1 \frac{\beta}{N} \eta_A \tilde{x}_1 - \tilde{\lambda}_2 \,
\left( \frac {\beta}{N} \eta_A \,\tilde{x}_1+\alpha \right) - \tilde{\lambda}_4
\, \left( -\alpha-\mu \right) \\[0.2cm]
\displaystyle \dot{\tilde{\lambda}}_5 
=  -\tilde{\lambda}_1 \,\theta- \tilde{\lambda}_5 \, \left( -\theta-\mu \right).
\end{cases}
\end{equation*}
We solve the optimal control problem 
\eqref{eq:model:PreP:control}--\eqref{eq:admiss:control} 
numerically for concrete parameter values and initial conditions. 
The initial conditions are given by \eqref{eq:init:SICAE} 
and the HIV transmission parameters take the values $\beta = 0.582$, 
$\eta_C = 0.04$, and $\eta_A = 1.35$. We assume that the default PrEP rate 
is equal to $\theta = 0.001$ and the parameters $\omega$, $\rho$, $\phi$, 
$\alpha$ take the values given by Table~\ref{table:parameters:HIV:CV}. 
The weight constants take the values $w_1 = w_2 = 1$. We start by solving 
the optimal control problem without the mixed state control constraint 
and compare the extremals with the case where the fraction 
of individuals under PrEP is constant for all $t \in [0, 25]$ 
and equal to $\psi = 0.1$ and $\psi = 0.9$. 
\begin{figure}[!htb]
\centering
\includegraphics[width=0.85\textwidth]{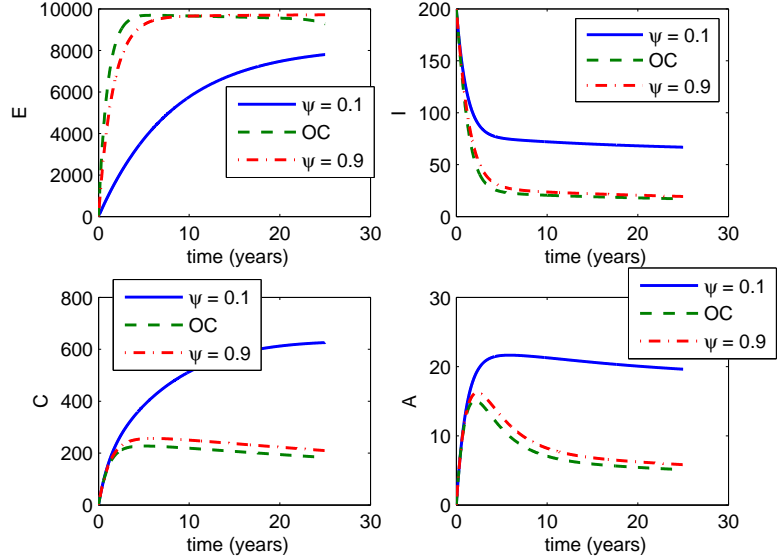}
\caption{Top left: Individuals under PrEP, $E$. 
Top right: pre-AIDS HIV infected individuals, $I$. 
Bottom left: HIV-infected individuals under ART treatment, $C$. 
Bottom right: HIV-infected individuals with AIDS symptoms, $A$.  
The continuous line is the solution of model \eqref{eq:model:PreP} 
for $\psi = 0.1$, the dashed line ``$-\, -$'' is the solution 
of the optimal control problem with no mixed state control constraint 
and ``$\cdot \, -$'' is the solution of model 
\eqref{eq:model:PreP} for $\psi = 0.9$. }
\label{fig:SICAE:control:NoMixed:psi:01:09}
\end{figure}
In Figure~\ref{fig:SICAE:control:NoMixed:psi:01:09}, we observe that the number 
of HIV-infected individuals associated with the optimal control solution 
$\tilde{I}$, $\tilde{C}$, $\tilde{A}$ is lower than the respective number 
associated with the constant values $\psi = 0.1$ or $\psi = 0.9$. On the other hand, 
we observe that the optimal control starts taking the maximum value $1$, 
which means that all susceptible individuals should be under PrEP 
(see Figure~\ref{Control_NoMixedPsi0109}) but this situation does 
not respect the mixed state control constraint (see Figure~\ref{uSnoMixed}) 
and implies higher implementation costs. 
\begin{figure}[!htb]
\centering
\subfloat[\footnotesize{Optimal control.}]{\label{Control_NoMixedPsi0109}
\includegraphics[width=0.45\textwidth]{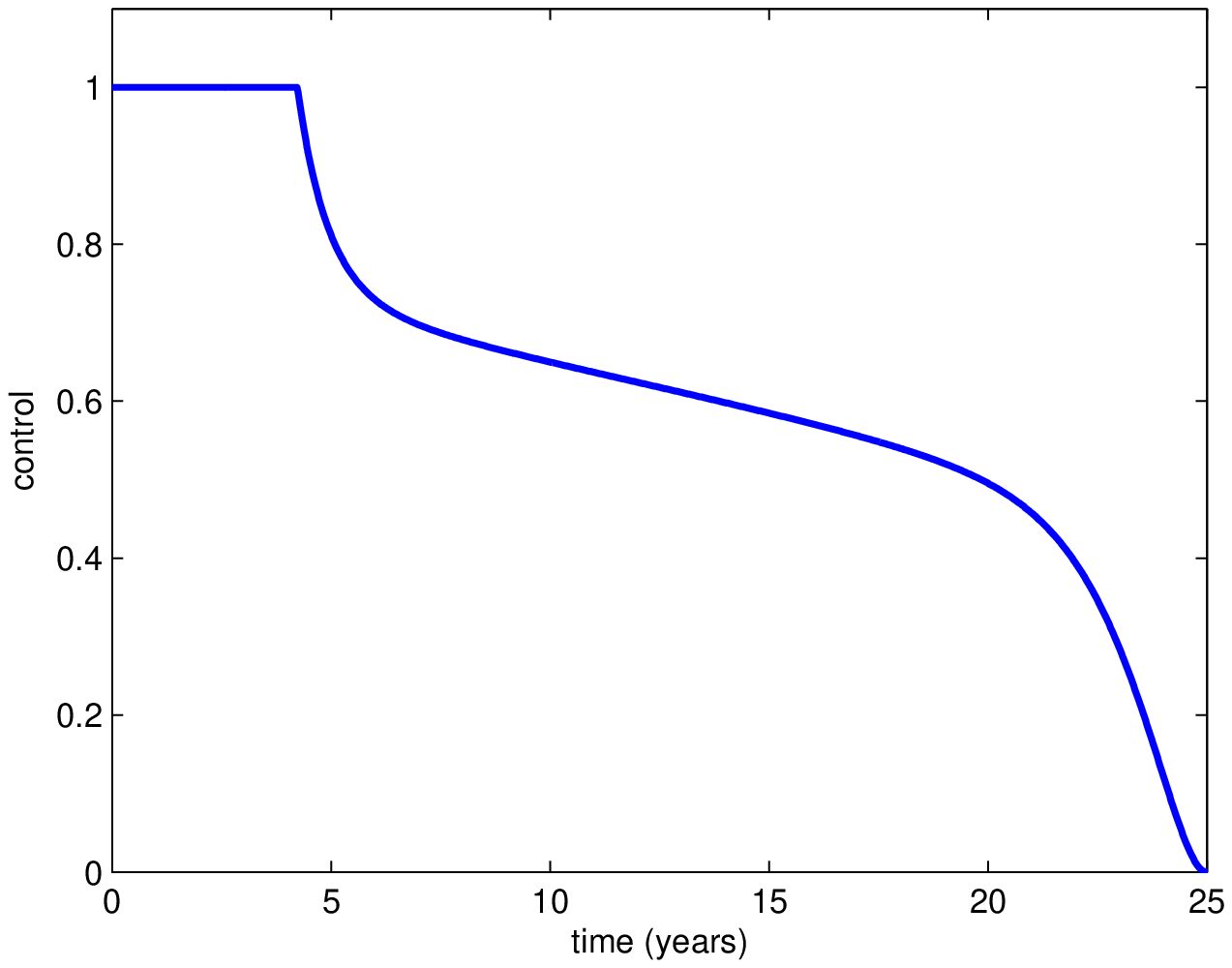}}
\subfloat[\footnotesize{Total individuals under PrEP.} ]{\label{uSnoMixed}
\includegraphics[width=0.45\textwidth]{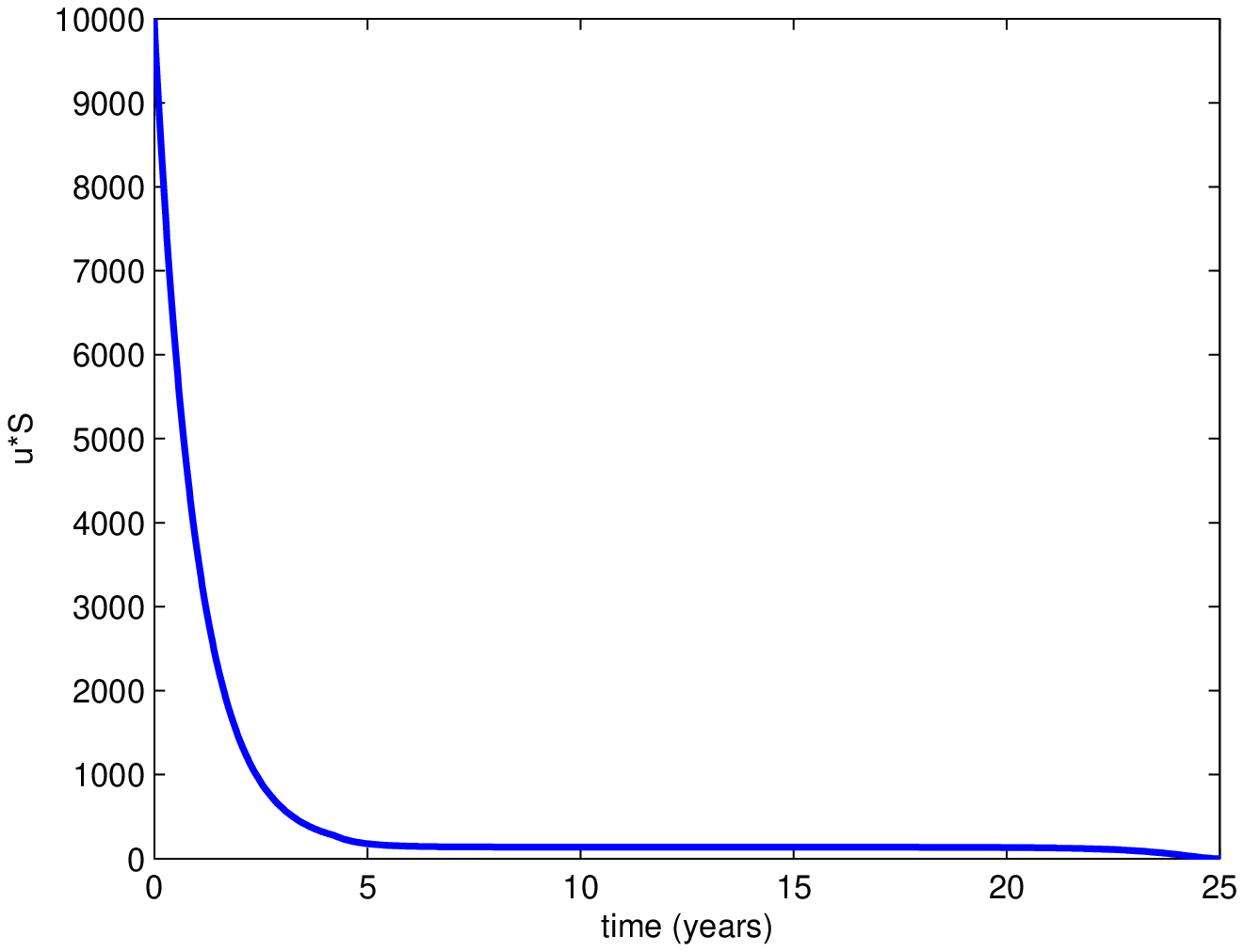}}
\caption{Solutions of the optimal control 
problem with no mixed state control constraint.
(a) Optimal control. (b) Total number of individuals 
that take PrEP at each instant of time.}
\label{control:uS:NoMixed}
\end{figure}
Now, consider the mixed state-control constraint 
\begin{equation}
\label{eq:constraint:num}
S(t) u(t) \leq 2000 \, , \quad \vartheta \geq 0 \, , \, \,
\text{for almost all} \, \, t \in [0, t_f].
\end{equation}
We start by comparing the optimal control solutions $\tilde{I}$, $\tilde{C}$, $\tilde{A}$ 
and $\tilde{E}$ associated with the optimal control $\tilde{u}$ with the solution 
of model \eqref{eq:model:PreP} with fixed values for the fraction of susceptible 
individuals under PrEP, $\psi = 0.1, 0.9$. 
\begin{figure}[!htb]
\centering
\includegraphics[width=0.85\textwidth]{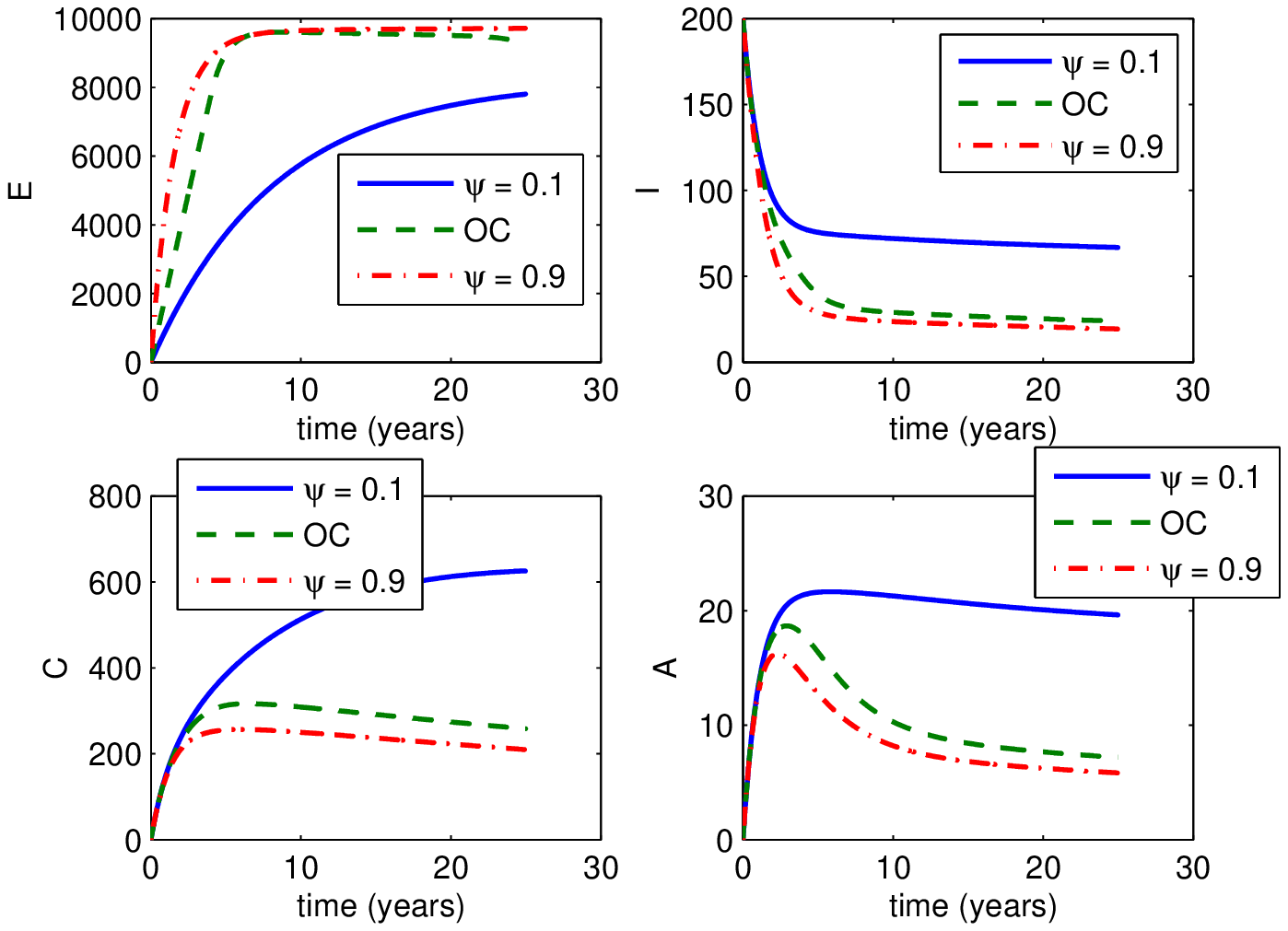}
\caption{Top left: Individuals under PrEP, $E$. 
Top right: pre-AIDS HIV infected individuals, $I$. 
Bottom left: HIV-infected individuals under ART treatment, $C$. 
Bottom right: HIV-infected individuals with AIDS symptoms, $A$.  
The continuous line is the solution of model \eqref{eq:model:PreP} 
for $\psi = 0.1$, the dashed line ``$-\, -$'' is the solution 
of the optimal control problem with the mixed state control constraint 
\eqref{eq:constraint:num} and ``$\cdot \, -$'' 
is the solution of model \eqref{eq:model:PreP} for $\psi = 0.9$.}
\label{fig:SICAE:control:psi:01:09}
\end{figure}
In Figure~\ref{fig:SICAE:control:psi:01:09}, we observe that optimal 
control solutions are not associated with the lowest values of the number 
of individuals with HIV infection. In fact, if we consider 
the case where 90 per cent of the susceptible population 
is under PrEP, the number of individuals with HIV infection 
is lower than the corresponding values associated with the optimal control solution. 
This is related to the mixed constraint \eqref{eq:constraint:num}. 
In Figure~\ref{uS}, we observe that the mixed constraint \eqref{eq:constraint:num} 
is satisfied for all $t \in [0, 25]$. The optimal control $\tilde{u}$ starts 
with the value $0.2$ and is an increasing function for $t \in [0, 4.05]$ years; 
for $t \in [4.05, 6.88]$, the optimal control $\tilde{u}$ takes the maximum value one,
which means that all susceptible individuals should be taking PrEP. 
For the rest of the simulation period, the optimal control $\tilde{u}$ 
is a decreasing function, taking values less than one. 
The mean value of the optimal control is approximately $0.61$
(see Figure~\ref{control}). 
\begin{figure}[!htb]
\centering
\subfloat[\footnotesize{Optimal control $\tilde{u}$}]{\label{control}
\includegraphics[width=0.33\textwidth]{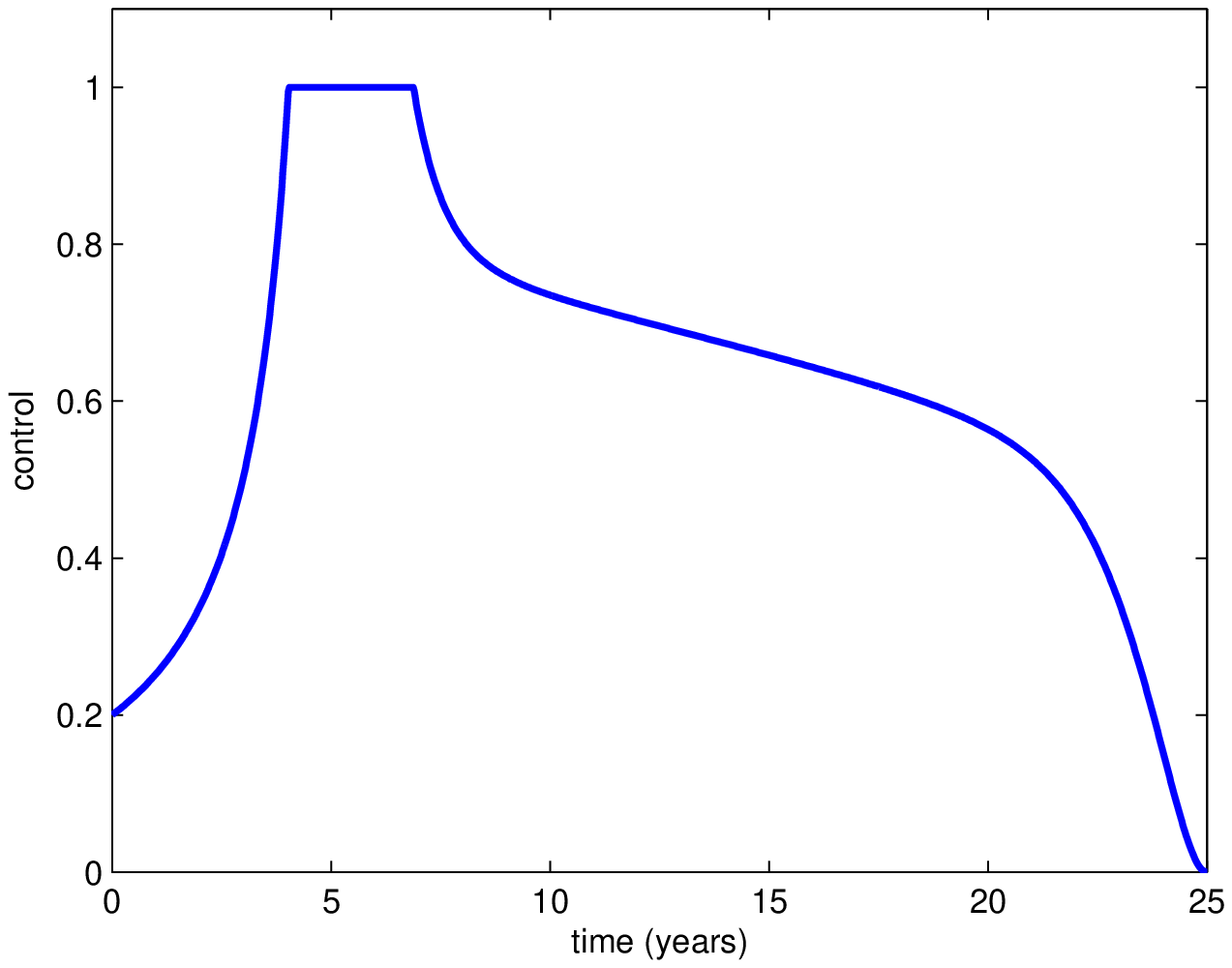}}
\subfloat[\footnotesize{Total individuals under PrEP.}]{\label{uS}
\includegraphics[width=0.33\textwidth]{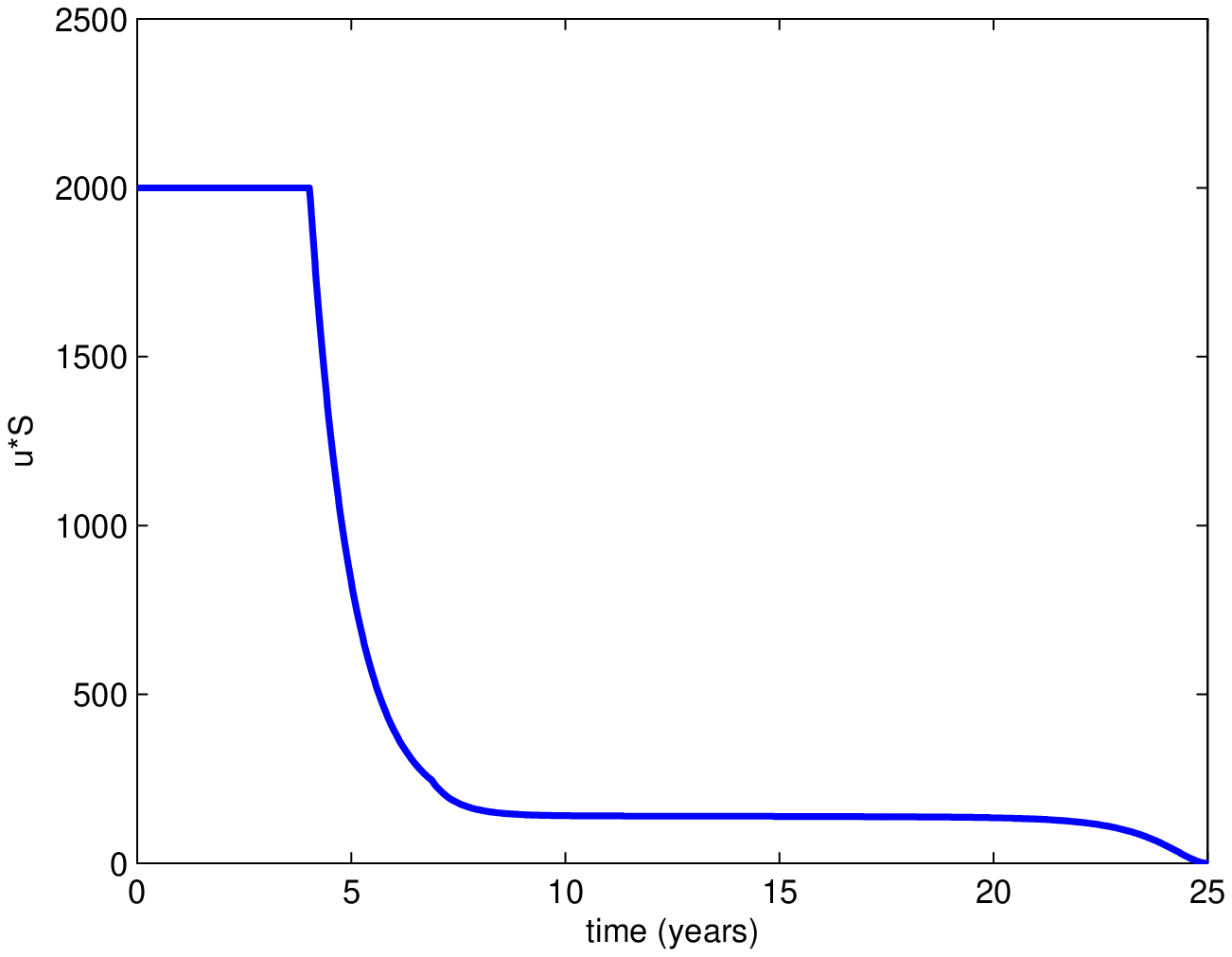}}
\subfloat[\footnotesize{Total individuals under PrEP ($\psi = 0.61$).}]{\label{psiS}
\includegraphics[width=0.33\textwidth]{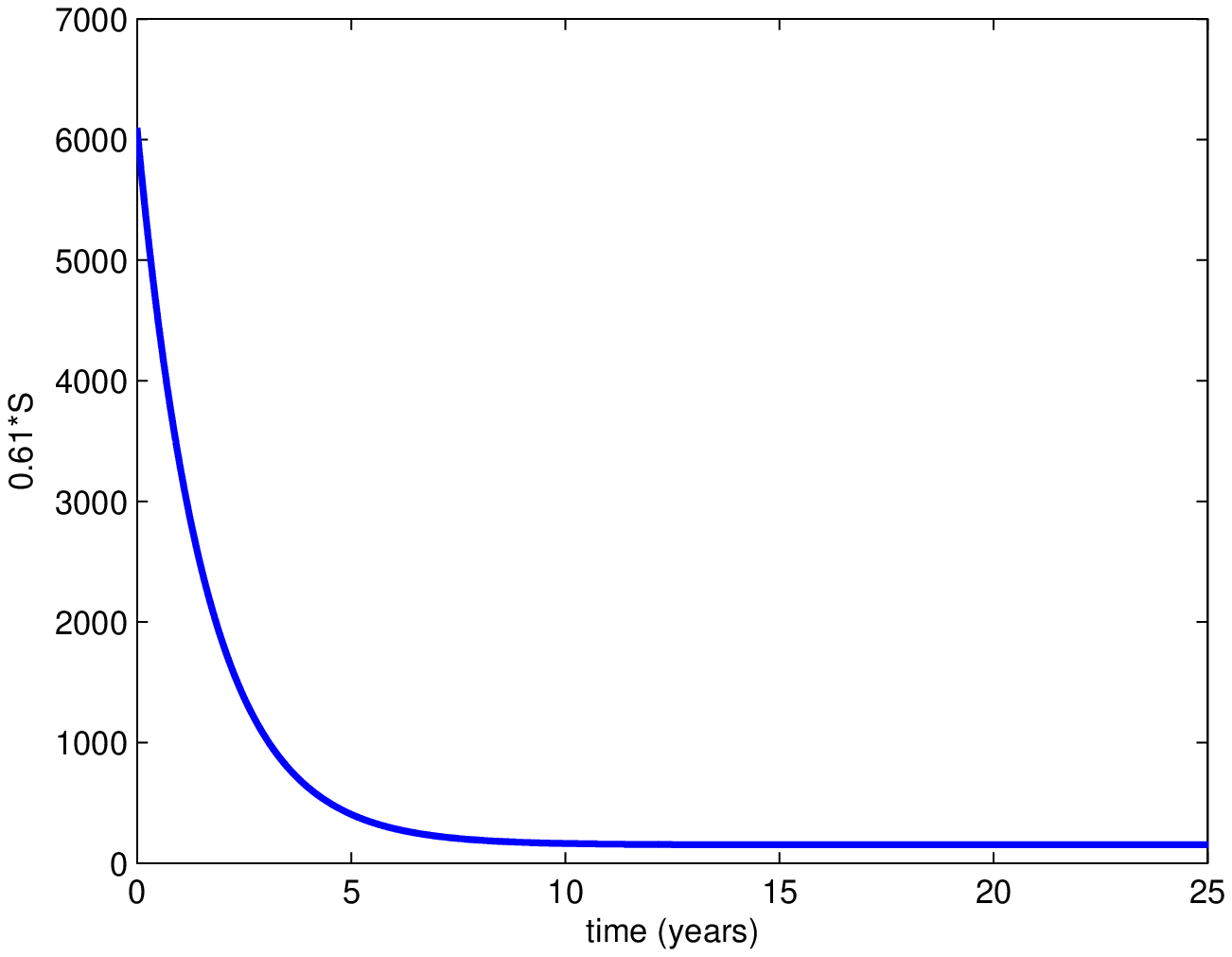}}
\caption{(a) Optimal control $\tilde{u}$ considering the mixed state control 
constraint \eqref{eq:constraint:num}. (b) Total number of individuals under PrEP 
at each instant of time for $t \in [0, 25]$ associated with the optimal control $\tilde{u}$. 
(c) Total number of individuals under PrEP at each instant of time 
for $t \in [0, 25]$ associated with $\psi = 0.61$.}
\label{fig:control:uS:psiS}
\end{figure}
In Figure~\ref{fig:SICAE:control:psi:06}, we simulate the case where 
the fraction of susceptible individuals under PrEP is constantly equal 
to the mean value of the optimal control $\tilde{u}$, that is, 
$\psi = 0.61$, and compare the solutions of model \eqref{eq:model:PreP} 
with the solutions of the optimal control problem with mixed state control 
constraint. We observe that the number of HIV infected individuals 
is lower for $\psi = 0.61$ than the ones associated with the optimal 
control solutions. However, the total number of individuals under PrEP 
at each instant of time is much bigger for $\psi = 0.61$ 
(see Figure~\ref{psiS}). In fact, we have $\int_0^{25} \tilde{u}(t) \tilde{S}(t) dt 
= 12553$ and $\int_0^{25} 0.61 \, S(t) dt = 13201$, 
where $S(t)$ denotes the solution of model \eqref{eq:model:PreP} 
for $\psi = 0.61$. We conclude that the cost associated with 
the PrEP strategy $\psi = 0.61$ is bigger than the one associated 
with the optimal control strategy. On the other hand, for 
the optimal control problem with no mixed state control constraint, 
we have $\int_0^{25} \tilde{u}(t) \tilde{S}(t) dt = 12836$, 
which means that this solution has a bigger cost 
than the one that satisfies the mixed constraint. 
\begin{figure}[!htb]
\centering
\includegraphics[width=0.85\textwidth]{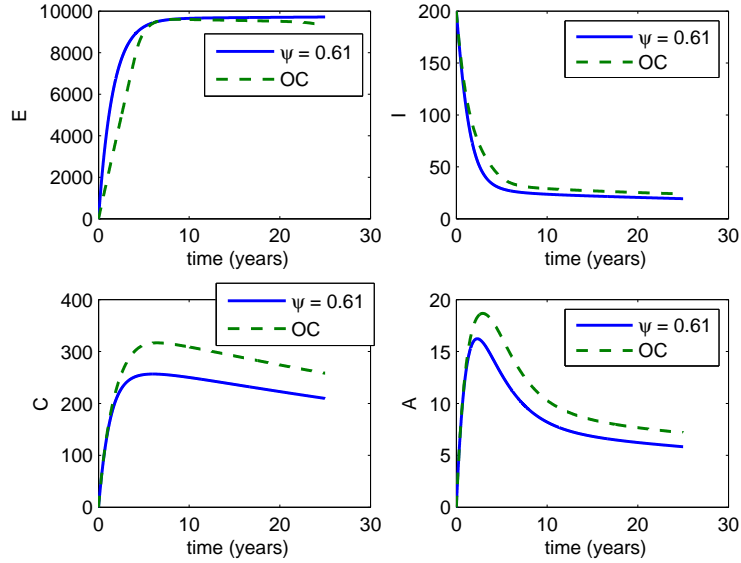}
\caption{Extremals of the optimal control problem 
\eqref{eq:model:PreP:control}--\eqref{eq:admiss:control}
with $\theta = 0.001$.}
\label{fig:SICAE:control:psi:06}
\end{figure}

Let us now suppose that we increase the value of $w_2$, 
that is, the weight associated with the cost of submitting 
susceptible individuals to PrEP. An increase of $w_2$ 
is associated with a decrease of the number of individuals 
that will be under PrEP. For example, for $w_2 = 10000$, 
the maximum value of the optimal control $\tilde{u}$ is approximately $0.17$ 
and $\int_0^{25} \tilde{u}(t) \tilde{S}(t) dt = 6652.5$, that is, 
the cost associated with the PrEP strategy decreases significantly. 
However, the number of HIV-infected individuals increases, namely 
$\tilde{I}(25) \simeq 110$, $\tilde{C}(25) \simeq 813$ and $\tilde{A}(25) \simeq 29$.   


\section{Conclusion and future work}
\label{sec:conclusion}

We have proved the global stability of the endemic equilibrium point 
of the SICA model for HIV/AIDS transmission proposed in \cite{SilvaTorres:TBHIV:2015}, 
in the case where the AIDS-induced death rate is negligible. Two different values 
for the relative infectiousness of HIV-infected individuals under ART treatment 
were considered, based on the studies reported in \cite{Cohen:NEJM:2011,DelRomero:2016} 
($\eta_C = 0.04$ and $\eta_C = 0.015$) and two values for the relative infectiousness 
of HIV-infected individuals with AIDS symptoms ($\eta_A = 1.3$ and $\eta_A = 1.35$), 
based in \cite{art:viral:load}. We have shown that taking the parameter values given 
in Table~\ref{table:parameters:HIV:CV}, the SICA model describes well the cumulative 
cases of infection by HIV and AIDS in Cape Verde for the period from 1987 to 2014
\cite{report:HIV:AIDS:capevert2015,WorldBank:TotalPop:url}. Furthermore, 
we generalized the SICA model to a HIV/AIDS-PrEP model (SICAE model) 
including PrEP as an HIV prevention strategy. We proved the existence 
and uniqueness of disease-free and endemic equilibrium points, 
depending on the value of the basic reproduction number $R_0$. Through Lyapunov functions 
and LaSalle's Invariance Principle, we proved the global stability of the disease-free 
equilibrium for $R_0 < 1$ and the global stability of the endemic equilibrium point 
when $R_0 > 1$ for negligible AIDS-induced death rate and strict adherence to PrEP.
The number of HIV-infected individuals for the SICA and SICAE models was compared.
We concluded that PrEP reduces the number of new HIV infections. However, 
only people who are HIV-negative and at very high risk for HIV infection should take PrEP. 
Therefore, the number of individuals that should take PrEP must be limited 
at each instant of time, for a fixed interval of time.  In order to study 
this health public problem, we formulated an optimal control problem 
with a mixed state control constraint, where the objective is to determine 
the optimal PrEP strategy that satisfies the mixed constraint 
(the total number of individuals under PrEP at each instant of time is limited) 
and minimizes the number of pre-AIDS HIV-infected individuals as well as the cost 
associated with the implementation of PrEP. The optimal control problems 
with and without the mixed state control constraint 
were solved and compared with the optimal solutions 
in the case where the fraction of individuals under PrEP is constant. 
Optimal control theory gave us PrEP strategies that 
minimize the number of HIV-infected individuals, the cost associated with PrEP 
and satisfy the limitations on the number of total individuals 
that should be under PrEP at each instant of time. 

It remains open the questions of how to prove the global stability of the endemic 
equilibrium point of the SICA model for positive AIDS-induced death rate 
and of the SICAE model for positive AIDS-induced death and PrEP default rates. 
As for the optimal control problem with mixed state control constraint, 
we believe it will be also interesting to consider a $L^1$ cost functional 
and a variable total population size. These and other questions are under
investigation and will be addressed elsewhere.
 

\section*{Acknowledgments}

This research was partially supported by 
Portuguese Foundation for Science and Technology (FCT)
through the R\&D unit CIDMA, reference UID/MAT/04106/2013,
and by project PTDC/EEI-AUT/2933/2014 (TOCCATA),
funded by FEDER funds through COMPETE 2020 -- 
Programa Operacional Competitividade e
Internacionaliza\c{c}\~ao (POCI) 
and by national funds through FCT.
Silva is also supported by the FCT post-doc 
fellowship SFRH/BPD/72061/2010.
The authors are grateful to two anonymous referees
for several pertinent questions and suggestions.



\medskip
Received September 2016; revised February and March 2017.
\medskip


\end{document}